\documentclass[12pt]{amsart}
\usepackage{ amsmath, amsthm, amsfonts, amssymb, color}
 \usepackage{mathrsfs}
\usepackage{amsfonts, amsmath}
 \usepackage{amsmath,amstext,amsthm,amssymb,amsxtra}
 \usepackage{txfonts} 
 \usepackage[colorlinks, citecolor=blue,pagebackref,hypertexnames=false]{hyperref}
 \allowdisplaybreaks
 \usepackage{pgf,tikz}
 \usepackage{multirow}
 \usepackage{diagbox} 

 \usepackage{tikz}
 \usepackage{cite}

\usetikzlibrary{decorations.pathreplacing}

\begin{document}

 \baselineskip 16.6pt
\hfuzz=6pt

\widowpenalty=10000

\newtheorem{cl}{Claim}
\newtheorem{theorem}{Theorem}[section]
\newtheorem{proposition}[theorem]{Proposition}
\newtheorem{coro}[theorem]{Corollary}
\newtheorem{lemma}[theorem]{Lemma}
\newtheorem{definition}[theorem]{Definition}
\newtheorem{assum}{Assumption}[section]
\newtheorem{example}[theorem]{Example}
\newtheorem{remark}[theorem]{Remark}
\renewcommand{\theequation}
{\thesection.\arabic{equation}}

\def\SL{\sqrt H}

\newcommand{\mar}[1]{{\marginpar{\sffamily{\scriptsize
        #1}}}}

\newcommand{\as}[1]{{\mar{AS:#1}}}

\newcommand\R{\mathbb{R}}
\newcommand\RR{\mathbb{R}}
\newcommand\CC{\mathbb{C}}
\newcommand{\Cn}{\mathbb{C}^n}
\newcommand\NN{\mathbb{N}}
\newcommand\ZZ{\mathbb{Z}}
\newcommand\HH{\mathbb{H}}
\newcommand\Z{\mathbb{Z}}
\def\RN {\mathbb{R}^n}
\renewcommand\Re{\operatorname{Re}}
\renewcommand\Im{\operatorname{Im}}

\newcommand{\mc}{\mathcal}
\newcommand\DD{\mathbb{D}}
\def\hs{\hspace{0.33cm}}
\newcommand{\la}{\alpha}
\def \l {\alpha}
\newcommand{\eps}{\tau}
\newcommand{\pl}{\partial}
\newcommand{\supp}{{\rm supp}{\hspace{.05cm}}}
\newcommand{\x}{\times}
\newcommand{\lag}{\langle}
\newcommand{\rag}{\rangle}

\newcommand\wrt{\,{\rm d}}

\newcommand{\norm}[2]{|#1|_{#2}}
\newcommand{\Norm}[2]{\|#1\|_{#2}}

\title[Singular values of Hankel operators]{IDA function and asymptotic behavior of singular values of Hankel operators on weighted Bergman spaces}

\author{Zhijie Fan}
\author{Xiaofeng Wang$^*$}\thanks{$*$ Corresponding Author}
\author{Zhicheng Zeng$^*$}

\address{Zhijie Fan, School of Mathematics and Information Science,
Guangzhou University, Guangzhou 510006, China}
\email{fanzhj3@mail2.sysu.edu.cn}

\address{Xiaofeng Wang, School of Mathematics and Information Science,
Guangzhou University, Guangzhou 510006, China}
\email{wxf@gzhu.edu.cn}

\address{Zhicheng Zeng, School of Mathematics and Information Science,
Guangzhou University, Guangzhou 510006, China}
\email{zhichengzeng@e.gzhu.edu.cn}

  \date{\today}



\begin{abstract}
In this paper, we use the non-increasing rearrangement of ${\rm IDA}$ function with respect to a suitable measure to characterize the asymptotic behavior of the singular values sequence $\{s_n(H_f)\}_n$ of Hankel operators $H_f$ acting on a large class of weighted Bergman spaces, including standard Bergman spaces on the unit disc, standard Fock spaces and weighted Fock spaces. As a corollary, we show that the simultaneous asymptotic behavior of $\{s_n(H_f)\}$ and $\{s_n(H_{\bar{f}})\}$ can be characterized in terms of the asymptotic behavior of non-increasing rearrangement of mean oscillation function. Moreover, in the context of weighted Fock spaces, we demonstrate the Berger-Coburn phenomenon concerning the membership of Hankel operators in the weak Schatten $p$-class.
\end{abstract}

\maketitle

\tableofcontents

\section{Introduction}
\subsection{Background}
We denote by $\text{Hol}(\Omega)$ the space of all holomorphic functions on a domain $\Omega$ of the complex plane $\mathbb{C}$.  Denote by $dA$ the Lebesgue measure and $dA_\omega:= \omega dA$, where $\omega$ is a suitable weight on $\Omega$. Then the weighted Bergman space $A^2 _\omega $ is given by
$$
	A^2_{\omega}:= \left\{ f \in \text{Hol} (\Omega):\ \ \| f \| _{A^2_{\omega}} := \left ( \displaystyle \int _\Omega |f(z)|^2dA_\omega(z) \right )^{1/2} <\infty\right\}.
	$$
In particular, if $\omega(z) =\frac{(\alpha+1)}{\pi}(1-|z|^2)^\alpha$ with $\alpha >-1$ and $\Omega=\mathbb{D}$ (unit disc), then $A^2 _\omega $ goes back to the standard Bergman space, which will be denoted by $A_\alpha^2$ for simplicity. If $\omega(z) =e^{-|z|^2}$ and $\Omega=\mathbb{C}$, then $A^2 _\omega $ goes back to the standard Fock space $F^2$.

Let $f\in\mathcal{S}$ (see \eqref{defS} for precise definition), then the Hankel operator with symbol $f$ is the operator $H_{f}: A_\omega^2\to L_\omega^2 \ominus A^2_\omega$ given by
\begin{equation*}
H_{f} g:= f g-P_\omega (f g),
\end{equation*}
where  $P_\omega$ is the orthogonal projection from $L^2_\omega:=L^2(\Omega, dA_\omega)$ onto $A^2_{\omega}$.

Hankel operators play a crucial role in various fields of mathematics, particularly in functional analysis, complex analysis, operator theory and control theory. In particular, they are among the most significant operators in the study of bounded and compact operators on Hilbert spaces and are deeply connected to many classical problems in analysis (see the book of V. Peller \cite{P2} for more backgrounds and applications about Hankel operators).

A fundamental problem in the study of Hankel operators is the characterization of their boundedness, compactness and Schatten-$p$ class membership (see e.g. \cite{MR2806550,MR850538,MR1113389,MR970119,MR2138695,MR4402674,MR4668087,MR3803293,MR3158507,MR1050105,MR1194989,MR3551773,MR3010276,MR1049650,MR1860488,MR1951248,ZWH,MR1013987,MR1087805,MR2311536,MR4552558,MR2104282} and the references therein). The complete characterizations of these three aspects of Hankel operators on the standard Bergman space were established by D.H. Luecking in 1992 \cite{MR1194989}. However, despite significant efforts and progress over the years (see e.g. \cite{MR2138695,MR3010276,MR3803293,MR2104282}), the analogous problems on Fock spaces remained open problems for over 20 years, until Z. Hu and J.A. Virtanen provided a complete answer in their recent remarkable works \cite{MR4402674,MR4668087}. In particular, under a priori assumption $f\in \mathcal{S}$, they introduced a new space ${\rm IDA}$ and demonstrated that
\begin{enumerate}
  \item[(A)]  $H_f$ is bounded from $F^2$ to $L^2(\mathbb{C},e^{-|z|^2}dA)$ if and only if $f\in BDA(\mathbb{C})$;
  \item[(B)]  $H_f$ is compact from $F^2$ to $L^2(\mathbb{C},e^{-|z|^2}dA)$ if and only if $f\in VDA(\mathbb{C})$;
  \item[(C)]  $H_f\in S^p(F^2\rightarrow L^2(\mathbb{C},e^{-|z|^2}dA))$ if and only if $f\in IDA^p(\mathbb{C})$, where $0<p<\infty$,
\end{enumerate}
where the definitions of $BDA$, $VDA$ and $IDA^p$ spaces will be recalled and extended to a more general context soon after.

Along the line carried out by D.H. Luecking \cite{MR1194989} and then by Z. Hu--J.A. Virtanen \cite{MR4402674,MR4668087}, another natural and fundamental question is how to characterize the asymptotic behavior of the singular values of $H_{f}$. However, unlike the extensively studied properties of boundedness, compactness, and Schatten class membership, this fundamental property has been largely unexplored. While some significant works (see \cite{MR2541276,MR3246989,MR2040916,MR4413302}) has been done on studying this fundamental property, a complete solution to this question remains unclear in both Bergman spaces and Fock spaces. Among these works, we would like to highlight a pioneering work of M. Bourass, O. El-Fallah, I. Marrhich and H. Naqos \cite{MR4413302}. To state their result, we first recall some notions in \cite{MR4413302}. Let $\mathcal{W}^*(\mathbb{D})$ denote a class of weights (see Section \ref{pre} for its precise definition) that includes many important examples, such as those associated with standard Bergman spaces, harmonically weighted Bergman spaces, and large Bergman spaces. Let $K$ be the reproducing kernel of $A^2_\omega$. Define  $K_z(w):=K(w,z)$ and let
\begin{align}\label{tauw}
\tau _\omega (z) := \frac{1}{\omega ^{1/2}(z)\|K_z \|_{A^2_{\omega}}}\quad \mbox{ and}\quad  d\lambda_\omega := \frac{dA}{\tau _\omega ^2 }.
\end{align}
It is also worth recalling from \cite[Section 2.1]{MR4413302} that $\tau _\omega ^2(z) \asymp \frac {1} {\Delta \varphi(z)}$ in several special contexts, where we write $\omega=e^{-\varphi}$. In particular, in the context of the standard Bergman space, $\Delta \varphi(z) \asymp \frac{1}{(1-|z|^2)^2}$, while in the context of the standard Fock space, $\Delta \varphi(z) \asymp 1$. Next, we recall the definitions of distribution function and non-increasing rearrangement function (see {\rm\cite[Section 1.4]{MR3243734}} for more details).
\begin{definition}
Let $f$ be a complex-valued function defined on $\Omega$. We denote by $f^*$ the non-increasing rearrangement of the function $f$ with respect to $d\lambda _\omega$.  Namely,
$$f^*(t):=\inf\{s>0:d_f(s)\leq t\},\quad t\geq 0,$$
where  $d_f(t):= \lambda _\omega( \{ z \in \Omega: \ |f(z)|> t   \} )$ is the distribution function.
\end{definition}
\noindent In the context of $\mathcal{W}^*(\mathbb{D})$-weighted Bergman space $A^2 _\omega $ and under a priori assumption that $f$ is an anti-analytic symbol (i.e. $f=\bar{\phi}$ for some $\phi\in \text{Hol}(\mathbb{D})$), M. Bourass et al. \cite{MR4413302} obtained the following results:

\begin{enumerate}
  \item[(i)] if $\rho $ is an increasing function  such that $\rho (x)/x^ \gamma$ is decreasing for some $\gamma\in (0,1)$, then
\begin{align}\label{firstequi}
s_n(H_f)\lesssim 1/\rho(n),\ {\rm for}\ {\rm all}\ n\in\mathbb{N} \Longleftrightarrow (\tau_\omega|\bar{f}'|)^*(n)\lesssim 1/\rho(n),\ {\rm for}\ {\rm all}\ n\in\mathbb{N};
\end{align}
  \item[(ii)] if $\rho $ is an increasing function  such that $\rho (x)/x^ \gamma$ is decreasing for some $\gamma>0$, then
\begin{align}\label{analooo}
 (\tau_\omega|\bar{f}'|)^*(n)\lesssim 1/\rho(n),\ {\rm for}\ {\rm all}\ n\in\mathbb{N} \Longrightarrow s_n(H_f)\lesssim 1/\rho(n),\ {\rm for}\ {\rm all}\ n\in\mathbb{N},
\end{align}
\end{enumerate}
where $s_n(H_f)$ denotes the $n$-th singular value of $H_f$ (see Section \ref{SV}). The converse direction of (ii) remains unclear but interested, particularly when the decay function $\rho$ is chosen as $(1+x)^{1/p}$ for $0<p\leq 1$. Their results \cite[Theorem 1.2 and 1.3]{MR4413302} indicate that in the setting of standard Bergman space $A_\alpha^2$,
\begin{enumerate}
  \item[(iii)]  if $0<p< 1$, then
    \begin{align*}
s_n(H_f)\lesssim \frac{1}{(1+n)^{1/p}},\ {\rm for}\ {\rm all}\ n\in\mathbb{N} \Longleftrightarrow f\ {\rm is}\ {\rm a}\ {\rm constant};
\end{align*}
  \item[(iv)]
   for the critical point $p=1$, we have
  \begin{align*}
s_n(H_f)\lesssim \frac{1}{1+n},\ {\rm for}\ {\rm all}\ n\in\mathbb{N} \Longleftrightarrow \bar{f}'\in H^1(\mathbb{D})\Longleftrightarrow  (\tau_\omega|\bar{f}'|)^*(n)\lesssim \frac{1}{1+n},\ {\rm for}\ {\rm all}\ n\in\mathbb{N},
\end{align*}
\end{enumerate}
where $H^1(\mathbb{D})$ denotes the Hardy space on $\mathbb{D}$, and the second equivalence, although not stated explicitly in their statement, is included in their proof.
\subsection{Our results}
Note that anti-analyticity is a strong condition imposed on the symbols, while $\mathcal{S}$ is a highly general class of symbols, designed to ensure that $H_f$ is densely defined on $A^2_\omega$. Therefore, the aim of this article is to establish a necessary and sufficient condition for the left-hand side of \eqref{firstequi}, which can be applied to a broader class of symbols $f\in\mathcal{S}$. For greater generality and broader applicability, we will work primarily in the framework of $W^*(\Omega)$-weighted Bergman space as in \cite{MR4413302}, such that our results encompass a wider range of important examples. Additionally, for completeness, we establish the boundedness, compactness, and Schatten-$p$ class characterizations $(1\leq p<\infty)$ in this framework as byproducts, which extends the results (A)--(C) to a broader setting.

Inspired by the works of D.H. Luecking \cite{MR1194989} and Z. Hu--J.A. Virtanen \cite{MR4402674,MR4668087}, our first step is to extend the notion of the ${\rm IDA}$ function $G_{\delta,\omega}(f)$ to the $W^*(\Omega)$-weighted setting, which can be regarded as a natural evolution from BMO (see \cite{MR4668087,MR131498,MR2934601}).

\begin{definition}
Let $B_\omega$ and $\delta_\omega$ be the constants chosen in Lemma \ref{lattice} later.
For any $\delta \in (0,B_\omega\delta _\omega)$ and $f\in L^2_{\mathrm{loc}}(\Omega)$ (the set of all locally square integrable functions on $\Omega$), we define the IDA function $G_{\delta,\omega}(f)$ as
\begin{equation}\label{G-q-r}
	G_{\delta,\omega}(f)(z):=\inf_{h\in {\rm Hol}(D(z,\delta \tau_\omega(z)))} \left(\frac{1}{A(D(z,\delta \tau_\omega(z)))}\int_{D(z,\delta \tau_\omega(z))} |f-h|^2 dA \right)^{\frac{1}{2}},\quad z\in \Omega,
\end{equation}
where $$D(z,\delta \tau_\omega(z)):= \{w\in\Omega : |z-w|<\delta \tau_\omega(z)\}.$$
\end{definition}

\begin{definition}
Fix $\delta\in (0,\delta _\omega)$. For $0<p\leq \infty$, the spaces $IDA_\omega^p(\Omega)$ (Integral Distance to Analytic Functions) and $IDA_\omega^{p,\infty}(\Omega)$ are defined as the set of $f\in L_{\rm loc}^2(\Omega)$ such that
$$\|f\|_{IDA_\omega^p(\Omega)}:=\|G_{\delta,\omega}(f)\|_{L^p(\Omega,d\lambda_\omega)}<+\infty,$$
$$\|f\|_{IDA_\omega^{p,\infty}(\Omega)}:=\|G_{\delta,\omega}(f)\|_{L^{p,\infty}(\Omega,d\lambda_\omega)}<+\infty,$$
where $L^{p,\infty}(\Omega,d\lambda_\omega)$ is the weak $L^p$ space consisting of all measurable function $f$ such that
$$\|f\|_{L^{p,\infty}(\Omega,d\lambda_\omega)}:=\sup\limits_{\lambda>0}\lambda d_f(\lambda)^{1/p}=\sup_{\lambda>0}\lambda^{1/p}f^*(\lambda),$$
where the second equality can be found in {\rm \cite[Proposition 1.4.9]{MR3243734}}.
For simplicity, we denote $BDA_\omega(\Omega)$ as $IDA_\omega^\infty(\Omega)$. The space $VDA_\omega(\Omega)$ consists of all $f\in  L^2_{\mathrm{loc}}(\Omega)$ such that $$\lim_{z\rightarrow\partial_\infty\Omega}G_{\delta,\omega}(f)(z)=0,$$
where $\partial_\infty\Omega$ is defined in \eqref{boundarydef}.
\end{definition}
In the setting of Fock space, the parameters $\delta_\omega$ and $\tau_\omega$ in the above definition can be usually replaced by $+\infty$ and $1$, respectively. In this setting, we use the notations $IDA^p(\mathbb{C})$ and $G_\delta(f)$ to replace $IDA_\omega^p(\mathbb{C},dA)$ and $G_{\delta,\omega}(f)$, respectively, for brevity.
The definitions of the above spaces are independent of $\delta$ (see \cite{MR4668087} for a related discussion).
The concept of $BDA$ ({\it Bounded Distance to Analytic Functions}) was introduced by D.H. Luecking \cite{MR1194989} in the context of Bergman space, while the spaces $IDA^p$ with $p<\infty$ were introduced recently in \cite{MR4668087} in the context of Fock space. In this paper, we extend the definition of $IDA^p$ to its weak-type version, and also extend these definitions to $\mathcal{W}^*(\Omega)$-weighted Bergman spaces. We will demonstrate the roles of these extensions in characterizing the behavior of singular values of Hankel operators on these weighted spaces. We now present our first main result.
\begin{theorem}\label{main0}
Let $\omega \in \mathcal{W}^\ast(\Omega)$, $\delta\in (0,\delta _\omega)$ and $f\in\mathcal{S}\cap VDA_\omega(\Omega)$. We have the following conclusions:
\begin{enumerate}
  \item\label{AAAAA} if $\rho $ is an increasing function  such that $\rho (x)/x^ \gamma$ is decreasing for some $\gamma\in (0,1)$. Then
 \begin{align}\label{ourform2}
 s_n(H_{f} ) \lesssim 1/ \rho (n),\ {\rm for}\ {\rm all}\ n\in\mathbb{N} \Longleftrightarrow  (G_{\delta,\omega}(f))^*(n)  \lesssim 1/ \rho (n),\  {\rm for}\ {\rm all}\ n\in\mathbb{N};
\end{align}
  \item\label{BBBBB} if $\rho $ is an increasing function  such that $\rho (x)/x^ \gamma$ is decreasing for some $\gamma>0$. Then
 \begin{align}\label{ourform33}
(G_{\delta,\omega}(f))^*(n)  \lesssim 1/ \rho (n),\  {\rm for}\ {\rm all}\ n\in\mathbb{N} \Longrightarrow  s_n(H_{f} ) \lesssim 1/ \rho (n),\ {\rm for}\ {\rm all}\ n\in\mathbb{N}.
\end{align}
\end{enumerate}

\end{theorem}
\begin{remark}{\rm
The condition $f\in VDA_\omega(\Omega)$ is imposed to ensure that $H_{f}$ is a compact operator, so that $s_n(H_{f} )$ is well-defined. Indeed, we will show in Proposition \ref{compactness} (resp. Proposition \ref{boundedness}) that $H_f$ is compact (resp. bounded) from $A^2_\omega$ to $L_\omega^2$ if and only if $f\in VDA_\omega(\Omega)$ (resp. $f\in BDA_\omega(\Omega)$).
}
\end{remark}

\begin{remark}\label{equiremark}{\rm
If $f\in VDA_\omega(\Omega)\subset BDA_\omega(\Omega)$, then  $(G_{\delta,\omega}(f))^*(0)\asymp s_0(H_f)<+\infty$. Hence, the left-hand side in \eqref{ourform33} is equivalent to the fact that $$(G_{\delta,\omega}(f))^*(t)  \lesssim 1/ \rho (t),\ {\rm for}\ {\rm all}\ t\geq0.$$}
\end{remark}
\begin{remark}\label{introremark}{\rm
The equivalence \eqref{ourform2} and the implication \eqref{ourform33} are even new for Hankel operators on both standard Bergman spaces and standard Fock spaces. However, note that these two parts are formulated in terms of $G_{\delta,\omega}(f)$, whereas the counterparts in \eqref{firstequi} and \eqref{analooo} are formulated in terms of $\tau_\omega|\bar{f}'|$. Therefore, it is necessary to prove that these two results are compatible, which will be achieved in Proposition  \ref{compareProp2}, by showing that the left-hand sides in \eqref{analooo} and \eqref{ourform33} are equivalent when the symbol $f$ is an anti-analytic function. Moreover, combining Theorem \ref{main0} with Remark \ref{keyremark000}, we will see that when $f$ is not an anti-analytic function, $G_{\delta,\omega}(f)$ is more suitable to be a characterization function than $\tau_\omega|\bar{f}'|$.}
\end{remark}

\begin{definition}
For any $\delta\in (0,B_\omega\delta_\omega)$ and $f\in L^2_{\textup{loc}}(\Omega)$, the mean oscillation of $f$ at $z$ is defined by
		$$MO_{\delta,\omega} (f)(z):=\left(\frac{1}{A(D(z,\delta \tau_\omega(z)))}\int_{D(z,\delta \tau_\omega(z))}|f-\hat{f}_\delta(z)|^2\,dA\right)^{1/2},\quad z\in \Omega,$$
where $\hat{f}_\delta$ is the averaging function of $f$ over disk $D(z,\delta\tau_\omega(z))$ by
$$\hat{f}_\delta(z):=\frac{1}{A(D(z,\delta \tau_\omega(z)))}\int_{D(z,\delta\tau_\omega(z))}fdA,\quad z\in\Omega.$$
In the setting of Fock space, we use the notation $MO_{\delta} (f)$ to replace $MO_{\delta,\omega} (f)$ for brevity.
\end{definition}
The simultaneous boundedness, compactness and membership in the Schatten-$p$ class of the Hankel operators $H_f$ and $H_{\bar{f}}$ have been extensively studied in various settings (see e.g. \cite{MR1860488,MR1951248,MR1087805,MR2311536,MR3551773,MR3158507,ZWH,MR4668087,MR4402674}). However, the study of simultaneous asymptotic behavior of the sequences $\{s_n(H_f)\}$ and $\{s_n(H_{\bar{f}})\}$ is still in its infancy. As a corollary of Theorem \ref{main0}, we characterize their simultaneous asymptotic behavior in terms of the asymptotic behavior of $(MO_{\delta,\omega} (f))^*$. The result is stated precisely as follows.
\begin{theorem}\label{simultaneous}
Let $\omega \in \mathcal{W}^\ast(\Omega)$, $\delta\in (0,\delta _\omega)$ and $f\in\mathcal{S}$. We have the following conclusions:
\begin{enumerate}
  \item if $\rho $ is an increasing function  such that $\rho (x)/x^ \gamma$ is decreasing for some $\gamma \in (0,1)$, then
      \begin{align*}
s_n(H_{f})+s_n(H_{\bar{f}}) \lesssim 1/\rho(n), \ {\rm for}\ {\rm all}\ n\in\mathbb{N} \Longleftrightarrow
(MO_{\delta,\omega}(f))^*(n) \lesssim 1/\rho(n), \ {\rm for}\ {\rm all}\ n\in\mathbb{N};
\end{align*}
  \item if $\rho $ is an increasing function  such that $\rho (x)/x^ \gamma$ is decreasing for some $\gamma >0$, then
      \begin{align*}
(MO_{\delta,\omega}(f))^*(n) \lesssim 1/\rho(n), \ {\rm for}\ {\rm all}\ n\in\mathbb{N}\Longrightarrow  s_n(H_{f})+s_n(H_{\bar{f}}) \lesssim 1/\rho(n), \ {\rm for}\ {\rm all}\ n\in\mathbb{N}.
\end{align*}
\end{enumerate}
\end{theorem}

Our next main result explores the converse direction of Theorem \ref{main0} \eqref{BBBBB}. Motivated by the results in \cite[Theorem 1.2 and 1.3]{MR4413302}, we mainly focus on the case when the decay function $\rho$ is chosen as $(1+x)^{1/p}$ (equivalently, $S^{p,\infty}$ characterization of Hankel operator) for $0<p\leq 1$. See Section \ref{SV} for the definition of $S^{p,\infty}$, and \cite{FLLXarxiv,MR4654013,FSZ,MR4755443,LMSZ,MSX2018,MSX2019,RS} for discussions on the role of $S^{p,\infty}$ and related study in analysis. Denote by $L^2(\varphi):= L^2(\mathbb{C}, e^{-\varphi}dA)$, and let $F^2(\varphi)$ be the weighted Fock space (see the third example in Section \ref{pre} for its definition). For simplicity, we formulate our result only in the setting of $F^2(\varphi)$ (see also Proposition \ref{Berezin} for an equivalent characterization via Berezin transform), although our argument is expected to also be applicable in other settings, such as the standard Bergman space $A_\alpha^2$.
\begin{theorem}\label{fastdecay}
Let $f\in \mathcal S$ and suppose that $\varphi\in C^2(\mathbb{C})$ is real-valued with $\mathrm{i} \partial \bar{\partial} \varphi \simeq \omega_0$, where $\omega_0=i\partial\bar{\partial}|z|^2$ is the Euclidean-K\"{a}hler form on $\mathbb{C}$. Then for any $0<p<\infty$, we have
\begin{enumerate}
  \item $H_f\in S^{p,\infty}(F^2(\varphi) \to L^2(\varphi))$ if and only if $f\in IDA^{p,\infty}(\mathbb{C})$;
  \item Both $H_f$ and $H_{\bar{f}}$ are in $S^{p,\infty}(F^2(\varphi) \to L^2(\varphi))$ if and only if $MO_\delta(f)\in L^{p,\infty}(\mathbb{C})$ for some $\delta\in (0,\infty)$.
\end{enumerate}

\end{theorem}
Under the same condition as Theorem \ref{fastdecay}, a complete Schatten-$p$ characterization of $H_f$ was established in the remarkable work of \cite{MR4402674}. Theorem \ref{fastdecay} extends this result to the weak-type version. Furthermore, it is noteworthy that, unlike the case of anti-analytic symbols, there is no cut-off point phenomenon in our general setting of $f\in\mathcal{S}$. Indeed, from the counterexample given in \eqref{counter000}, we see that there exists a function $f\in IDA^{p,\infty}(\mathbb{C})$ for all $0<p<\infty$ such that $f$ is not a constant.

In 1987, C.A. Berger and L.A. Coburn \cite{MR882716} proved the following celebrated result: for $f\in L^\infty(\mathbb{C}^m)$, $H_f$ is compact if and only if $H_{\bar{f}}$ is compact. After a personal discussion with L.A. Coburn, the authors in \cite{MR2104282} posed a natural open problem: for $f\in L^\infty(\mathbb{C}^m)$ and $1\leq p<\infty$, is it true that $H_f\in S^p$ if and only if $H_{\bar{f}}\in S^p$? W. Bauer \cite{MR2063120} was the first to establish this property for Hilbert-Schmidt Hankel operators on $F^2$. No further progress was made for 16 years after his paper, until recently Z. Hu and J.A. Virtanen provided a complete answer in \cite{MR4402674} (see also \cite{MR4552558,MR4630767}). They showed that the conclusion holds for $1<p<\infty$ but fails for $0<p\leq 1$. This property is referred to as the Berger-Coburn phenomenon for the Schatten class $S^p$. It is noteworthy that an analogous statement fails on Bergman spaces (see \cite{MR4373140}) and generally does not hold if the symbol $f$ is unbounded (see \cite{MR2063120}). Inspired by these observations, our next result explores the Berger-Coburn phenomenon for the weak Schatten-$p$ class $S^{p,\infty}$ on the weighted Fock space $F^2(\varphi)$.
\begin{theorem}\label{Coburn}
Suppose that $\varphi\in C^2(\mathbb{C})$ is real-valued and satisfies $i\partial\bar{\partial}\varphi\asymp \omega_0$, where $\omega_0=i\partial\bar{\partial}|z|^2$ is the Euclidean-K\"{a}hler form on $\mathbb{C}$. Let $f\in L^\infty(\mathbb{C})$, then for any $1<p<\infty$,
$$H_f\in S^{p,\infty}(F^2(\varphi) \to L^2(\varphi)) \Longleftrightarrow H_{\bar{f}}\in S^{p,\infty}(F^2(\varphi) \to L^2(\varphi)).$$
\end{theorem}
\subsection{Difficulties and our strategy}
Compare to the special setting of anti-analytic symbols discussed in \cite[Theorem 1.1]{MR4413302}, the main difficulty we encounter in characterizing the asymptotic behavior of $s_n(H_f)$ for $f\in \mathcal{S}$ is that
when $f\in\mathcal{S}$ is not an anti-analytic symbol, subharmonicity of $|\bar{f}'|$ can no longer be applied as in the proof of \cite[Theorem 1.1]{MR4413302}.

To overcome the aforementioned difficulty, our proof relies on key elements such as convex functions tricks, $\bar{\partial}$-techniques and decomposition theory of functions in $L^2_{\rm loc}(\Omega)$, combining some of the arguments given in \cite{MR4413302,MR3490780,MR4402674,MR4668087}. In particular, developed by \cite{MR4413302,MR3490780}, the convex function trick plays a key role in obtaining the estimates of singular values of compact operator $T$ from upper and lower bounds of ${\rm Trace}(h(T))$ for some suitable convex functions $h$. Moreover, as the canonical solution to $\bar{\partial}u=g\bar{\partial}f$, $H_fg$ is naturally linked closely to the $\bar{\partial}$-theory. Furthermore, we decompose $f\in \mathcal{S}\subset L_{\rm loc}^2(\Omega)$ into two components: one is connected closely with Carleson measures and the other can be handled by the technique of H\"{o}rmander $L^2$ estimates for the $\bar{\partial}$-operator (see e.g. \cite{MR1169879,MR1285493,MR1273537,MR1194989} for the early development of this technique by H. Li and D.H. Luecking in studying the properties of Hankel operators). While the combination of these techniques is sometimes straightforward,  it often requires significantly more effort.
\subsection{Organization of our paper}
The paper is organized as follows. Section \ref{Preli} provides preliminaries and background information, consisting of five parts: the first part recalls some basic notions of singular values, Schatten-$p$ classes and weak Schatten-$p$ classes; the second part recalls the notion of $W^*$ weights and constructs a partition of unity subordinated to a certain lattice; the third part recalls some basic concepts about Hankel operators and Toeplitz operators on weighted Bergman space; the fourth part establishes several technical lemmas, which are key ingredients in the proof of our main theorem; In Section \ref{Boundedness and compactness}, we establish the boundedness and compactness characterizations for Hankel operators in the context of $\mathcal{W}^*(\Omega)$-weighted Bergman spaces. In Section \ref{mainsection}, we first present the proof of Theorem \ref{main0}, followed by the proof of Theorems \ref{simultaneous} and \ref{fastdecay}. As a byproduct, we also establish the Schatten-$p$ ($1<p<\infty$) class characterization of Hankel operators in the $W^*(\Omega)$-weighted setting. In Section \ref{relationsection}, we carry out the task mentioned in Remark \ref{introremark} (see also the comments mentioned at the beginning of Section \ref{relationsection}). Finally, Section \ref{BCph} provides a proof of Theorem \ref{Coburn}.

\section{Preliminaries}\label{Preli}
\subsection{Notation}
Throughout the whole paper, we denote by $C$ a positive constant which is independent of the main parameters, but it may
vary from line to line. We use $A\lesssim B$ to denote the inequality that $A\leq CB$ for some constant $C>0$, and $A\asymp B$ to denote the inequality that $A\gtrsim B$ and $B\lesssim A$.


We denote the average of a function $f$ over a Lebesgue measurable set $E$ by
\begin{align*}
\fint_{E}fdA:=\frac{1}{A(E)}\int_{E}fdA.
\end{align*}
For two suitable functions $f$ and $g$ on $\mathbb{C}$, their convolution $f\ast g$ is defined by
\begin{align}\label{defconv}
(f\ast g)(z):=\int_{\mathbb{C}}f(z-w)g(w)dA(w).
\end{align}
We set $\mathbb{N}$ be the set of non-negative integers. The indicator function of a subset $E\subset \mathbb{C}$ is denoted by $\chi_{E}$. We use the notation $\# E$ to denote the number of elements in the finite set $E$.
For an operator $T$, we simply use the notation $\|T\|$ to denote its operator norm.
\subsection{Singular values}\label{SV}
Let $H$ be a complex Hilbert space and let $T$ be a compact operator on $H$. Note that $T^{*}T$ is compact, positive and therefore diagonalizable. The singular values sequence $\{s_n(T)\}_{n \geq 1}$ is defined as the non-increasing sequence of the square roots of the eigenvalues of $T^*T$, counted according to multiplicity. Equivalently, $s_n(T)$ can be characterized by
\begin{align*}
s_n(T)=\inf\{\|T-F\|:{\rm rank}(F)\leq n\}.
\end{align*}
If $T$ is positive, then $\{s_n(T)\}_{n\geq 0}$ is the sequence of eigenvalues of $T$. In this case, we write $s_n(T)=\lambda _n(T)$. For $0<p<\infty$, a compact operator $T$ is said to belong to the Schatten-$p$ class $S^p:=S^p(A^2_\omega \to L_\omega^2)$ if
$$\|T\|_{S^p}:=\left(\sum_{n\geq 0}s_n(T)^p\right)^{1/p}<+\infty.$$
Moreover, for $0<p<\infty$, a compact operator $T$ is said to belong to the weak Schatten-$p$ class $S^{p,\infty}:=S^{p,\infty}(A^2_\omega \to L_\omega^2)$ if the quasi-norm of $T$ satisfies
$$\|T\|_{S^{p,\infty}}:=\sup\limits_{n\geq 0}(1+n)^{1/p}s_n(T)<+\infty.$$
More details about Schatten class and weak Schatten class can be found in e.g. \cite{LSZ1,LSZ2}.

The following result is the well-known monotonicity Weyl's Lemma.
\begin{lemma}\label{Weyl's Lemma}
Let $T_1,T_2$ be two positive bounded operators on a complex Hilbert space $H$ such that $T_1\leq T_2$. If $T_2$ is compact, then $T_1$ is compact and $\lambda _n(T_1) \leq \lambda _n (T_2)$ for all $n\geq 1$.
\end{lemma}
\subsection{The class of weights $\mathcal{W}^\ast(\Omega)$}\label{pre}
This subsection is devoted to recalling the class of weights $\mathcal{W}^\ast$ introduced in \cite{MR4413302}. To begin with, we recall the definition of the class of  weights $\mathcal{W}(\Omega)$ introduced in \cite{MR3490780}, with conditions imposed on the reproducing kernel.
Let $\Omega$ be a domain (bounded or not) of $\mathbb{C}$ and let $\partial{\Omega}$ be the boundary of $\Omega$. Let
\begin{align}\label{boundarydef}
\partial_\infty\Omega:=\begin{cases}
\partial\Omega,& {\rm if}\ \Omega\ {\rm is}\ {\rm bounded},\\
\partial\Omega\cup\{\infty\},& {\rm if}\ \Omega\ {\rm is}\ {\rm not}\ {\rm bounded}.
\end{cases}
\end{align}
\begin{definition}
Let $\omega: \Omega\rightarrow (0, \infty)$ be a positive continuous weight, which is assumed to be bounded below by a positive constant on each compact set of $\Omega$. If $\omega$ satisfies the following hypothesis:
\begin{enumerate}\label{Con1}
  \item[${\rm\bf (H)_1}$] we have $$ \lim _{z\rightarrow \partial_\infty\Omega}\| K_z\|_{A^2_{\omega}} = \infty;$$
  \item[${\rm\bf (H)_2}$] for every  $\zeta \in \Omega$,  we have $$|K(\zeta, z)| = o(\|K_z\|_{A^2_{\omega}}), \quad {\rm as}\ z\rightarrow\partial_\infty\Omega;$$
  \item[${\rm\bf (H)_3}$] for every $z\in\Omega$, we have
  \begin{align*}
      \tau_\omega(z):=\frac{1}{\omega ^{1/2}(z)\|K_z \|_{A^2_{\omega}}}=O(\min\{1,{\rm dist}(z,\partial_\infty\Omega)\});
      \end{align*}
  \item[${\rm\bf (H)_4}$] there exists a constant $\eta >0$ such that for $z,\zeta \in \Omega$ satisfying $| z-\zeta| \leq \eta \tau_\omega (z)$, we have
	\begin{align*}
	\tau_\omega(z) \asymp \tau_\omega (\zeta)\ \mbox{and}\  \|K _z \| _{A^2_{\omega}} \|K _\zeta \| _{A^2_{\omega}} \lesssim |K(\zeta, z)|,
	\end{align*}
\end{enumerate}
 then we say that the weight $\omega$ belongs to the class $\mathcal{W}(\Omega)$.
\end{definition}

\begin{lemma}\label{lattice}{\rm \cite{MR3490780}}
Let $\omega \in \mathcal{W}(\Omega)$, then there exist $B_\omega>1 $ and $\delta_\omega \in (0, \eta/4B_\omega)$ such that for all $\delta \in (0,\delta _\omega)$, there exists a sequence $\{z_n\}\in \Omega$ such that
\begin{enumerate}
\item $\Omega =\displaystyle  \bigcup _{n\geq 1}D (z_n, \delta \tau _\omega (z_n))= \bigcup _{n\geq 1}D (z_n, B_\omega \delta \tau _\omega (z_n))$;\\
\item $D (z_n,  \frac{\delta}{B_\omega } \tau _\omega (z_n)) \cap D (z_m,  \frac{\delta}{B_\omega } \tau _\omega (z_m))= \emptyset $, for $n\neq m$;\\
\item $ z \in D (z_n, \delta \tau _\omega (z_n))$ implies that $D (z, \delta \tau _\omega (z)) \subset D (z_n, B_\omega \delta \tau _\omega (z_n))$;\\
\item There exists an integer $N$ such that every $D (z_n,B_\omega \delta \tau _\omega (z_n))$ cuts at most $N$ sets of the
family $\left\{D (z_m,B_\omega \delta \tau _\omega (z_m))\right \}_m$. We say that $\left\{D (z_n,B_\omega \delta \tau _\omega (z_n))\right\}_n$ is of finite multiplicity.
\end{enumerate}
\end{lemma}

Throughout the paper, unless otherwise specific, we fix the parameters mentioned in Lemma \ref{lattice}.
\begin{definition}
We say that $\{R_n\}_n \in  \mathcal{L}_\omega (\delta) $ if  $\{R_n\}_n:= \{ D (z_n, \delta \tau _\omega (z_n))\}_n$ satisfies the above conditions.
\end{definition}
To establish a suitable decomposition theory for $L^2_{\mathrm {loc}}(\Omega)$ function (see Lemma \ref{decomposelemma}), we need to construct a partition of unity subordinated to a certain lattice.
	\begin{lemma}\label{cover}
Let $\{R_n\}_n\in \mathcal{L}_\omega(\delta)$, then there exists a partition of unit $\{\psi_n\}$ subordinated to the covering $\{D(z_n,\delta \tau_\omega(z_n))\}_n$ of $\Omega$ with $\| \tau_\omega\bar{\partial}\psi_n\|_\infty\lesssim 1$ uniformly for each $n$.
	\end{lemma}
	\begin{proof}
		Let $f_n:[0,\infty)\to [0,1]$ be a compactly supported smooth function such that $\supp f_n\subset [0,2\delta \tau_\omega(z_n))$, $f_n(x)=1$ on $[0,\delta \tau_\omega(z_n))$ and $\|f'_n\|_\infty\leq 2\delta^{-1} \tau_\omega(z_n)^{-1}$. Let $g_n(z):=f_n(|z-z_n|)$, then $g_n(z)=1$ on $D(z_n,\delta \tau_\omega(z_n))$, $\supp g_n\subset D(z_n,2\delta \tau_\omega(z_n))$ and
		$\|\bar{\partial}g_n\|_\infty\leq 2\delta^{-1} \tau_\omega(z_n)^{-1}.$ Note that $\textup{supp}\,\bar{\partial}g_n\subset D(z_n,2\delta \tau_\omega(z_n))$ and $\tau_\omega(z)\asymp \tau_\omega(z_n)$ for all $z\in D(z_n,2\delta \tau_\omega(z_n))$. Thus, $\| \tau_\omega\bar{\partial}g_n\|_\infty\lesssim 1.$  Let $$\psi_n(z):=\frac{g_n(z)}{\sum_m g_m(z)}.$$
		Since $0\leq g_n(z)\leq 1$ and $1\leq \sum_m g_m(z)\leq N$ for all $z\in\Omega$, where $N$ is the integer chosen in Lemma \ref{lattice}, we obtain
		\begin{align*}
			|\bar{\partial}\psi_n(z)|=\left|\frac{\bar{\partial}g_n(z)\sum_m g_m(z)-g_n(z)\sum_m \bar{\partial}g_m(z)}{[\sum_m g_m(z)]^2}\right|
			\leq 2\sum_m |\bar{\partial}g_m(z)|.
		\end{align*}
Noting that the sum on the right-hand side contains no more than $N$ non-zero terms, we see that  $\| \tau_\omega\bar{\partial}\psi_n\|_\infty\lesssim 1.$ This ends the proof of Lemma \ref{cover}.
	\end{proof}
Let $\Delta$ be the Laplace operator given by $\Delta = \partial \bar{\partial}$, where
	$$\partial := \frac{1}{2}\Big(\frac{\partial }{\partial x}-i\frac{\partial }{\partial y}\Big),\quad  \bar{\partial }:=   \frac{1}{2}\Big(\frac{\partial }{\partial x}+i\frac{\partial }{\partial y}\Big).$$

		\begin{definition}
		Let $\omega= e^{- \varphi} \in\mathcal{W}(\Omega)$ such that $\varphi\in C^2(\Omega)$. We say that $\omega\in\mathcal{W}^\ast(\Omega)$ if $\omega$ satisfies
\begin{equation}\label{C9}
	\tau_\omega ^2(z) \Delta\varphi(z)\gtrsim 1, \quad z\in\Omega,
	\end{equation}
	or, there exist a subharmonic function $\psi: \Omega \to \mathbb{R}^+$ and constants $\delta>0$, $t\in (-1,0)$ such that
	\begin{equation}\label {C6}
	\tau_\omega ^2(z) \Delta\psi(z)\geq \delta,\ \Delta\varphi(z)\geq t \Delta \psi (z)\ \mbox{and}\ |\partial\psi(z)|^2\leq  \Delta \psi (z),\quad z\in \Omega.
	\end{equation}
	\end{definition}
The class of weights $\mathcal{W}^*(\Omega)$ covers many important examples. In the particular choice of $\Omega=\DD$, \cite[Section 2.1.1]{MR4413302} presents three important examples, including standard Bergman spaces, harmonically weighted Bergman spaces and large Bergman spaces. Moreover, when $\Omega=\mathbb{C}$, the following three examples also belong to $\mathcal{W}^*(\mathbb{C})$.
\begin{itemize}
  \item {\it Standard Fock spaces.} Let $\alpha >0$. The standard Fock spaces are defined as follows.
 \begin{equation*}
F ^2_\alpha := \left\{ f \in {\rm Hol} (\mathbb{C}): \  \displaystyle \int _{\mathbb{C}}|f(z)|^2 e ^{-\alpha |z|^2 }dA(z)<\infty \right\}.
 \end{equation*}
 Then the reproducing kernel is given by $K(z,w) = e ^{\alpha z{\bar w} }$ and $\tau_\omega(z)\asymp 1$. The standard Fock spaces are special cases of Fock type spaces and weighted Fock spaces introduced as the subsequent two examples.

  \item {\it Fock type spaces.} Let $\Psi: [0,+\infty)\to [0,+\infty)$ be a $C^3$ function such that
\begin{equation*}
\Psi'(x) >0, \quad \Psi ''(x) \geq 0\quad  {\rm and}\quad  \Psi '''(x)\geq 0, \quad x\in [0,+\infty).
\end{equation*}
Moreover, we assume that there exists a real number $\eta<1/2$ such that the function $\Phi (x):= x\Psi '(x)$ satisfies
\begin{equation*}
\Phi '' (x) = O(x^{-1/2}\left (\Phi '(x)\right )^{1+\eta} ),
\end{equation*}
The Fock type spaces $F^2_{\Psi}$ are defined as follows.
$$F^2_{\Psi}:= \left\{ f \in {\rm Hol} (\mathbb{C}): \ \displaystyle \int _{\mathbb{C}}|f(z)|^2 e ^{-\Psi(|z|^2)}dA(z)<\infty \right\}.$$
The Fock type spaces $F^2_{\Psi}$ were first considered by K. Seip and E.H. Youssfi in \cite{MR3010276}.
Recall from \cite{MR4516170} that $\omega:= e^{-\Psi (| z|^2)}\in\mathcal{W}(\mathbb{C})$ and
$$
\tau_\omega(z) \asymp \Phi '(|z|^2)^{-\frac12}.
$$
Furthermore, note that $\Delta(\Psi(|z|^2))=\Phi'(|z|^2)$. This verifies the condition \eqref{C9} and therefore, $\omega\in\mathcal{W}^*(\mathbb{C})$.

\item {\it Weighted Fock spaces.} Suppose that $\varphi\in C^2(\mathbb{C})$ is real-valued. Assume that there are two positive constants $m$ and $M$ such that
    \begin{align}\label{ddc}
    m\omega_0\leq i\partial\bar{\partial}\varphi\leq M\omega_0,
    \end{align}
    where $\omega_0=i\partial\bar{\partial}|z|^2$ is the Euclidean-K\"{a}hler form on $\mathbb{C}$.
    Then the weighted Fock spaces $F^2(\varphi)$ are defined as follows.
$$F^2(\varphi):= \left\{ f \in {\rm Hol} (\mathbb{C}): \ \displaystyle \int _{\mathbb{C}}|f(z)|^2 e ^{-\varphi(z)}dA(z)<\infty \right\}.$$
See e.g. \cite{MR4402674,MR4668087,MR2891634} and the reference therein for more background about these Fock spaces.
    By \cite[Lemma 2.2]{MR4402674}, it is easy to verify that $\omega:=e^{-\varphi(z)}\in\mathcal{W}(\mathbb{C})$ with $\tau_\omega(z)\asymp 1$. Furthermore, note that the condition \eqref{ddc} simplifies to the form $m\leq \Delta\varphi\leq M$, which verifies the condtion \eqref{C9}. Therefore, $\omega\in\mathcal{W}^*(\mathbb{C})$.
\end{itemize}

\subsection{Hankel operators and Toeplitz operators}
In this subsection, we recall some basic concepts about Hankel operators and Toeplitz operators on weighted Bergman space (see \cite{MR4413302,MR4516170,MR3490780} for more details). Let $A^2 _\omega $ be the weighted Bergman space as given in the introduction. Bergman space $A^2_\omega $ is a reproducing kernel Hilbert space. We denote its reproducing kernel by $K$. Let $K_z(w):=K(w,z)$ and let $k_z:= \frac{K_z}{\| K_z\|_{A_\omega^2}}$ denote the normalized reproducing kernel of  $A^2_\omega$.

We denote by $P_\omega$ the orthogonal projection from $L^2_\omega$ onto $A^2_\omega$, which can be represented as follows
$$
P_\omega (f) (z)= \displaystyle \int _\Omega f(\zeta)K(z,\zeta)dA_\omega(\zeta),\quad f\in L^2_\omega.
$$	
The domain of $P_\omega$ can be extended to those functions $f$ satisfying $f K_z\in L^1_\omega:=L^1(\Omega, dA_\omega)$ for all $z\in\Omega$.
\begin{definition}
Set
		$$\Gamma:=\left\{\sum _{1\leq i \leq n}c_i K_{z_i}: c_i \in \mathbb{C}, z_i \in \Omega, n\in \mathbb{N} \right\},$$
and
		\begin{align}\label{defS}
\mathcal{S}:=\left\{f\, \textup{is measurable on }\, \Omega: fg \in L^2_\omega\, \textup{ for }\, g\in \Gamma\right\}.
\end{align}
Let $f\in \mathcal{S}$, then the Hankel operator $H_f$ associated with symbol $f$ is given by
\begin{equation*}
	H_f g:= f g-P_\omega (f g),\quad g\in A_\omega^2.
	\end{equation*}
\end{definition}
Hankel operator $H_f$ is a densely defined operator on $A^2_\omega$, which can be expressed explicitly by the following formula:
	\begin{equation*}
	H_f g(z)= \int_\Omega (f(z)-f(w))g(w)K(z,w)dA_\omega(w),\quad z\in\Omega.
	\end{equation*}

Next, we recall the definition of Toeplitz operator.
\begin{definition}
Let $\mu$ be a positive Borel measure on $\Omega$, then the Toeplitz operator $T_\mu$ associated with $\mu $ and defined on $A^2_\omega $ is defined by
	$$
	T_\mu f(z) := \displaystyle \int_\Omega f(w) K(z,w)\omega (w)d\mu (w), \quad z\in \Omega.
	$$
\end{definition}
It is a standard fact that $T_\mu$ satisfies the following useful formula
\begin{align}\label{Toeformula}
\langle T_\mu f , f\rangle_{A_\omega^2} = \displaystyle \int_\Omega |f(z)|^2\omega (z)d\mu (z).
\end{align}

The Toeplitz operator $T_\mu$ is closely connected with the embedding operator $J_\mu$ defined by
\begin{align*}
J_\mu:\ A_\omega^2&\longrightarrow L_\omega^2(\mu)\\
f&\longmapsto f,
\end{align*}
where $L_\omega^2(\mu)$ is the space of $\mu$-measurable Borel functions on $\Omega$ such that
$$\int_{\Omega}|f(z)|^2\omega(z)d\mu(z)<+\infty.$$
Indeed, it can be verified that $T_\mu=J_\mu^* J_\mu$, where $J_\mu^*$ is the adjoint operator of $J_\mu$. Consequently, $T_\mu$ is bounded on $A_\omega^2$ if and only if $J_\mu$ is bounded on $A_\omega^2$.
\subsection{Technical Lemmas}
This subsection is devoted to establishing several technical lemmas, which play key roles in establishing our main theorem.

\begin{definition}
Suppose that $\delta\in(0,B_\omega\delta_\omega)$. We let $M_{\delta,\omega}$ be the average operator of order $2$ over the disk $D(z,\delta \tau_\omega(z))$ defined by
	\begin{align}\label{defM}
M_{\delta,\omega}(f)(z):=\left(\fint_{D(z,\delta \tau_\omega(z))}|f|^2\,dA\right)^{1/2},\quad z\in \Omega,
\end{align}
for all $f\in L^2_{{\rm loc}}(\Omega)$.
\end{definition}
Inspired by the idea given in \cite[Page 254-255]{MR1194989}, we can use the partition of unity to show the following decomposition lemma.
	\begin{lemma}\label{decomposelemma}
Let $\omega \in \mathcal{W}^\ast(\Omega)$, $\delta\in (0,\delta _\omega)$ and $f\in L^2_{\mathrm {loc}}(\Omega)$, then $f$ admits a decomposition $f=f_1+f_2$ such that $f_1\in C^\infty(\Omega)$ and
		$$\tau_\omega(z)\bar{\partial}f_1(z)+M_{\delta,\omega} (\tau_\omega \bar{\partial}f_1)(z)+M_{\delta,\omega} (f_2)(z)\lesssim G_{B_\omega^2\delta,\omega}(f)(z)$$
		for all $z\in \Omega.$
	\end{lemma}
	\begin{proof}
Let $\{\psi_n\}$ be a partition of unity given in Lemma \ref{cover}, which is subordinated to the covering $\{D(z_n,\delta \tau_\omega(z_n))\}_n$ of $\Omega$ with $\| \tau_\omega\bar{\partial}\psi_n\|_\infty\lesssim 1$ uniformly for each $n$. For every $n,m\geq 1$, we choose $h_{n,m}\in \text{Hol}(R_n)$ so that
		\begin{equation}\label{eq.1}
			G_{\delta,\omega}(f)(z_n)^2\leq \fint_{R_n}|f-h_{n,m}|^2\,dA\leq G_{\delta,\omega}(f)(z_n)^2+\frac{1}{m}.
		\end{equation}
This implies that
		\begin{align*}
			\fint_{R_n}|h_{n,m}|\,dA&\leq \left(\fint_{R_n}|h_{n,m}|^2\,dA\right)^{1/2}\\
			&\leq \left(\fint_{R_n}|f-h_{n,m}|^2\,dA\right)^{1/2}+\left(\fint_{R_n}|f|^2\,dA\right)^{1/2}\\
			&\leq G_{\delta,\omega}(f)(z_n)^2+\frac{1}{m}+\left(\fint_{R_n}|f|^2\,dA\right)^{1/2}.
		\end{align*}
This, together with the subharmonicity of function $|h_{n,m}|$, yields that $\{h_{n,m}\}_m$ is a normal family. Without loss of generality, this allows us to assume that $\{h_{n,m}\}_m$ converges to some $h_n\in \text{Hol}(R_n)$ uniformly on any compact subset of $R_n$ as $m\to \infty$. Letting $m\to \infty$ in \eqref{eq.1} and applying Fatou's Lemma, we conclude that
		\begin{equation}\label{eq.2}
			G_{\delta,\omega}(f)(z_n)^2= \fint_{R_n}|f-h_{n}|^2\,dA.
		\end{equation}
		Now we set
		$$f_1:=\sum_n h_n\psi_n\quad \textup{and}\quad f_2:=f-f_1.$$
		Clearly, we have $f_1\in C^\infty(\Omega)$. To continue, we will show that $|h_j-h_k|$ has a suitable upper bound on $R_j\cap R_k$. Indeed, by equality \eqref{eq.2} and the fact that $\tau_\omega(z)\asymp \tau_\omega(z_j)\asymp \tau_\omega(z_k)$ for $z\in R_j\cap R_k$,
		\begin{align*}
			|h_j(z)-h_k(z)|&\leq \left(\fint_{D(z,B_\omega^{-1}\delta\tau_\omega(z))}|h_j-h_k|^2\,dA \right)^{1/2}\\
			&\lesssim \left(\fint_{ R_j}|h_j-f|^2\,dA \right)^{1/2}+\left(\fint_{ R_k}|h_k-f|^2\,dA \right)^{1/2}\\
			&=G_{\delta,\omega}(f)(z_j)+G_{\delta,\omega}(f)(z_k).
		\end{align*}
		For any $z\in \Omega$, let $J_z$ be the set of integers $j$ such
		that $z\in R_j$, and let $j_0$ be an arbitrary element in $J_z$. We write
		$$f_1(z)=\sum_{j\in J_z} (h_{j_0}(z)+h_j(z)-h_{j_0}(z))\psi_j(z)=h_{j_0}(z)+\sum_{j\in J_z} (h_j(z)-h_{j_0}(z))\psi_j(z).$$
This, together with the fact that $h_j\in \text{Hol}(R_n)$ for any $j\geq 1$, yields
		\begin{align*}
			\tau_\omega(z)|\bar{\partial}f_1(z)|&= \tau_\omega(z)\left|\sum_{j\in J_z} (h_j(z)-h_{j_0}(z))\bar{\partial}\psi_j(z)\right|\\
			&\lesssim \sum_{j\in J_z} \big(G_{\delta,\omega}(f)(z_j)+G_{\delta,\omega}(f)(z_{j_0})\big).
		\end{align*}
		Since $R_j\subset D(z,B_\omega\delta \tau_\omega(z))$ for all $z\in R_j$,  we have $G_{\delta,\omega}(f)(z_j)\lesssim G_{B_\omega\delta,\omega}(f)(z)$. Combining this inequality with the fact that $|J_z|\leq N$ for any $z\in \Omega$, where $N$ is the constant given in Lemma \ref{lattice}, we conclude that $\tau_\omega(z)|\bar{\partial}f_1(z)|\lesssim G_{B_\omega\delta,\omega}(f)(z)$.
		This further implies that $M_{\delta,\omega}(\tau_\omega|\bar{\partial}f_1|)(z)\lesssim G_{B_\omega^2\delta,\omega}(f)(z)$.

Now we turn to $f_2$. To begin with, by Cauchy-Schwarz's inequality,
		$$|f_2|=\left|f-\sum_n h_n\psi_n\right|=\left|\sum_n (f-h_n)\psi_n\right|\leq \left(\sum_n |f-h_n|^2\psi_n\right)^{1/2}.$$
		This, together with equality \eqref{eq.2}, yields
		\begin{align*}
			M_{\delta,\omega}(f_2)(z)^2&\leq \fint_{D(z,\delta \tau_\omega(z))}\sum_n |f-h_n|^2\psi_n\,dA\\
			&\leq \sum_n\frac{1}{A(D(z,\delta \tau_\omega(z)))}\int_{D(z,\delta \tau_\omega(z))\cap R_n} |f-h_n|^2\,dA\\
			&\lesssim \sum_{\{n:D(z,\delta \tau_\omega(z))\cap R_n\neq \emptyset\}}G_{\delta,\omega}(f)(z_n)^2\\
			&\lesssim G_{B_\omega\delta,\omega}(f)(z)^2.
		\end{align*}
Since $G_{B_\omega\delta,\omega}(f)(z)\lesssim G_{B_\omega^2\delta,\omega}(f)(z)$ for all $z\in\Omega$, we finish the proof of Lemma \ref{decomposelemma}.
	\end{proof}

The following lemma is a generalization of $\bar{\partial}$-H\"{o}rmander's Theorem.
\begin{lemma}\label{dbar}
Let $\omega \in \mathcal{W}^\ast(\Omega)$, then there exists a constant $C>0$ such that for any function $g$ on $\Omega$, there exists a solution $u$ to the equation $\bar{\partial}u=g$ such that
	$$
	\int_{\Omega} |u(z)|^2 dA_\omega (z) \leq C \int_{\Omega}\tau_\omega ^2  (z) |g(z)|^2 dA_\omega(z).
	$$
\end{lemma}	
\begin{proof}
We refer the reader to \cite[Lemma 3.2]{MR4413302} for the proof. The statement there is given on $\mathbb{D}$, but the proof also applies to a general domain $\Omega$ of $\mathbb{C}$.
\end{proof}
\begin{coro}\label{keycoro}
Let $\omega \in \mathcal{W}^\ast(\Omega)$, then there exists a constant $C>0$ such that for any $g\in \Gamma$ and $f\in \mathcal{S}\cap C^1(\Omega)$,
\begin{align*}
			\|H_{f}g\|_{L_\omega^2}^2&\leq C \int_\Omega \tau_\omega(z)^2|\bar{\partial}f(z)|^2|g(z)|^2 dA_\omega(z).
		\end{align*}
\end{coro}
\begin{proof}
		By Lemma \ref{dbar}, for any $g\in \Gamma$ and $f\in \mathcal{S}\cap C^1(\Omega)$, there exists a solution $u$ to the equation
		\begin{equation}\label{d-bar0}
			\bar{\partial}u= g\bar{\partial}f
		\end{equation}
		satisfying
		\begin{equation}\label{dbarinequa}
			\int_\Omega |u(z)|^2dA_\omega(z) \lesssim \int_\Omega \tau_\omega(z)^2|\bar{\partial}f(z)|^2|g(z)|^2 dA_\omega(z).
		\end{equation}

We claim that $H_{f}g$ is the $L^2_\omega-$minimal solution to the equation \eqref{d-bar0}. Indeed, since $P_\omega(fg)\in \text{Hol}(\Omega)$ and $g\in \text{Hol}(\Omega)$, we have
		\begin{align}\label{bbarr}
\bar{\partial}(H_{f}g)=\bar{\partial}(fg-P_\omega(fg))=g\bar{\partial}f.
\end{align}
Let $u$ be an arbitrary solution of $\bar{\partial}$-equation \eqref{d-bar0} satisfying the inequality \eqref{dbarinequa}, then we decompose $u$ as $u=(u-H_{f}g)+H_{f}g$. It follows from \eqref{d-bar0} and \eqref{bbarr} that $u-H_{f}g\in A^2_\omega$. This, in combination with the fact that $H_{f}g\in L_\omega^2 \ominus A^2_\omega$, implies that the decomposition is orthogonal. Therefore, 	$$\|u\|_{L_\omega^2}^2=\|u-H_{f}g\|_{L_\omega^2}^2+\|H_{f}g\|_{L_\omega^2}^2\geq \|H_{f}g\|_{L_\omega^2}^2.$$
This shows that $H_{f}g$ is the $L^2_\omega-$minimal solution to the equation \eqref{d-bar0}. Therefore, the proof of Corollary \ref{keycoro} is completed.
\end{proof}

Next, following the ingenious argument provided in \cite[Lemma 3.1]{MR4516170}, we establish a technical lemma involving convex functions, which is a useful tool to obtain estimates of singular values of compact operator $T$ from upper and lower bounds of ${\rm Trace}(h(T))$ for some suitable convex functions $h$. To begin with, we let $\theta >0$ and introduce a convex function $h_{\theta}$ defined on $[0,\infty)$ as follows:
$$ h_\theta(t):= (t-\theta)^+:= \max (t-\theta,0) .$$

\begin{lemma}\label{Convex1}
Let $a,b: [0,+\infty) \rightarrow (0,+\infty)$ be decreasing functions such that $$\displaystyle \lim _{t\to \infty} a(t)=\lim _{t\to \infty} b(t) = 0.$$  \begin{enumerate}
  \item Suppose that there exists $\gamma \in (0,1)$ such that $t\rightarrow t^\gamma b(t)$ is increasing. Suppose that there exist constants $ B,C >0$ such that for every $\theta>0$,
\begin{align}\label{convex123}
\displaystyle \sum _{n\geq 1} h_\theta(a(n)) \leq \displaystyle C\int_0^\infty h_\theta (Bb(t))dt,
\end{align}
then $a(n) \lesssim b(n)$ for all integer $n\geq 1$. Furthermore, if the expression on the left-hand side is replaced by $\int_0^\infty h_\theta(a(t))dt$, then $a(t)\lesssim b(t)$ for all $t\geq 0$.
  \item Suppose that there exists $\gamma \in (0,1)$ such that the sequence $\{n^\gamma a(n)\}_{n\geq 1}$ is increasing. Suppose that there exist constants $ B,C >0$ such that for every $\theta>0$,
\begin{align}\label{convex123000}
\displaystyle \int_0^\infty h_{\theta}\left (\frac{1}{B}b(t)\right )dt \leq  \displaystyle C\sum _{n\geq 1} h_\theta(a(n)),
\end{align}
then $b(n) \lesssim a(n)$ for all integer $n\geq 1$. Furthermore, if the expression on the right-hand side is replaced by $\int_0^\infty h_\theta(a(t))dt$, then $b(t)\lesssim a(t)$ for all $t\geq 0$.
\end{enumerate}

%
\end{lemma}
\begin{proof}
We only provide a proof for the first statement, since the other ones are similar. Without loss of generality, we suppose that the functions $a$ and $b$ are strictly decreasing. It follows from inequality \eqref {convex123} that
\begin{align}\label{AI1}
 \displaystyle \sum _{a(n) \geq  \theta }a(n)  \leq  2 \displaystyle \sum _{a(n) \geq \theta }\big(a(n) -\frac{\theta}{2}\big)   \leq 2 \displaystyle \sum _{n\geq 1 }h_{\frac{\theta}{2}}(a(n) ) \leq 2C \displaystyle \int_{0}^\infty h_{\frac{\theta}{2}}(Bb(t) )dt\leq 2CB\displaystyle \int_{b(t)\geq \frac{\theta}{2B}} b(t)dt.
\end{align}

Let $\tilde{b}(t)$ be the inverse function of $t\rightarrow \frac{1}{b(t)}$, then we have
$$ \int_{b(t)\geq x}b(t)dt\asymp x\tilde{b}(\frac{1}{x}).$$
Indeed, on the one hand, it is direct that
$$
x\tilde{b}(\frac{1}{x})\lesssim  \int_{b(t)\geq x}b(t)dt.
$$
On the other hand, since $t\rightarrow t^\gamma b(t)$ is increasing,
$$\int_{b(t)\geq x}b(t)dt= \int_{b(t)\geq x}t^\gamma b(t)\cdot\frac{1}{t^\gamma}dt\lesssim x\big(\tilde{b}(\frac1x)\big)^{\gamma}\int_{b(t)\geq x}\frac{1}{t^{\gamma}}dt\asymp x\tilde{b}(\frac{1}{x}).$$
Substituting $x=\theta/2B$ into the above inequality and then applying \eqref{AI1}, we deduce that $$ \sum _{a(n) \geq \theta }a(n)  \lesssim \frac{\theta}{B}\tilde{b}\left(\frac{2B}{\theta}\right).$$

\noindent Let $N(\theta)$ be the number of integers $n$ such that $a(n)\geq \theta $, then $\theta N(\theta)\leq \sum_{a(n)\geq \theta}a(n)$. Furthermore,
\begin{equation}\label{Ntheta1}
 N(\theta)\lesssim \frac{1}{B}\tilde{b}\left(\frac{2B}{\theta}\right).
\end{equation}
Choosing $\theta=a(m)$ in the above inequality, we deduce that
\begin{align}\label{lowermaj}
m\lesssim \frac{1}{B}\tilde{b}(\frac{2B}{a(m)}).
\end{align}This implies that
$$a(m)\lesssim b(m).$$
This ends the proof of Lemma \ref{Convex1}.
\end{proof}

\begin{lemma}\label{finelemma}
Assume that $0<\gamma<p$.
Let $a:[0,+\infty) \rightarrow (0,+\infty)$ be a decreasing function such that $\displaystyle \lim _{t\to \infty} a(t)= 0,$ and let $\rho $ be an increasing function such that $\rho (x) / x^\gamma$ is decreasing. Suppose that
$$
 \displaystyle \sum _{n\geq 1} h(a(n)) \leq \displaystyle \int_{0}^\infty h (1/\rho (t))dt
$$
for all increasing function $h$ such that $h(t^{p})$ is convex. Then there is a constant $C$ depending on $p$ and $\gamma$ such that
$
a(n) \leq \frac{C}{\rho (n)}
$
for all integer $n\geq 1$. Furthermore, if the expression on the left-hand side is replaced by $\int_0^\infty h(a(t))dt$, then $a(t)\leq \frac{C}{\rho (t)}$ for all $t\geq 0$.
\end{lemma}
\begin{proof}
We only provide a proof for the first statement, since the second one is similar.
Without loss of generality, we suppose that the functions $a$ and $\rho$ are strictly decreasing and increasing, respectively.
Let $\delta >0$ and consider $h(t)= \left ( t^{1/p}-\delta ^{1/p}\right )^+$. By assumption,
$$
 \displaystyle \sum _{n\geq 1} \left (a(n)^{1/p}-\delta ^{1/p}\right )^+ \leq \displaystyle \int_{0}^\infty \left (\frac{1}{\rho(t) ^{1/p}}-\delta ^{1/p}\right )^+dt.
$$
Then we have
$$
\Big(1-\frac{1}{2^{1/p}}\Big) \displaystyle \sum _{a(n) \geq 2\delta} a(n)^{1/p}\leq \displaystyle \sum _{a(n) \geq 2\delta}(a(n)^{1/p}-\delta^{1/p})\leq \displaystyle \sum _{n\geq 1} \left (a(n)^{1/p}-\delta ^{1/p}\right )^+  \leq \displaystyle \int_{\rho (t) \leq 1/\delta }\frac{1}{\rho(t) ^{1/p}}dt.
$$
Using the fact that $x^{\gamma/p}/\rho ^{1/p}(x)$ is increasing and the fact that $\gamma/p <1$, we conclude that
$$
\delta ^{1/p}\#\{n:\ a(n)\geq 2\delta \}\lesssim\delta ^{1/p}\rho^{-1}(1/\delta).
$$
This implies the desired result and ends the proof.
\end{proof}
\section{Boundedness and compactness characterizations}\label{Boundedness and compactness}
This section is devoted to establishing the boundedness and compactness characterizations for Hankel operators in the context of $\mathcal{W}^*(\Omega)$-weighted Bergman spaces.
\begin{theorem}\label{boundedness}
Let $\omega \in \mathcal{W}^\ast(\Omega)$, $\delta\in (0,\delta _\omega)$ and $f\in\mathcal{S}$, then the Hankel operator $H_f$ is bounded from $A^2_\omega$ to $L_\omega^2$
	if and only if $f\in BDA_\omega(\Omega)$. Moreover, $\| H_f \| \asymp \displaystyle \|f\|_{BDA_\omega(\Omega)} $, where the implicit constants depend only on $\delta$ and $\omega$.
\end{theorem}
\begin{proof}
$(\Longrightarrow)$ We first suppose that $H_f$ is bounded from $A^2_\omega$ to $L_\omega^2$. Since
		$$
		\tau_\omega(z)= \frac{1}{\omega^{1/2}(z)\,\|K_z\|_{A^2_{\omega}}},
		$$
		we deduce that for $\xi \in D(z,\delta \tau_\omega(z))$,
		$$|k_{z}(\xi)|\asymp \|K_\xi\|_{A^2_{\omega}}= \tau_\omega(\xi)^{-1}\omega(\xi)^{-1/2}.$$
		It follows that $k_{z}^{-1}\in \text{Hol}(D(z,\delta \tau_\omega(z)))$, and that
		\begin{align}\label{uselater}
			\|H_f k_{z}\|_{L_\omega^2}^2&=\int_{\Omega} |fk_{z}-P_\omega(fk_{z})|^2\omega\,dA\nonumber\\
			&\geq \int_{D(z,\delta \tau_\omega(z))} |fk_{z}-P_\omega(fk_{z})|^2\omega\,dA\nonumber\\
			&=\int_{D(z,\delta \tau_\omega(z))} |k_{z}|^2\left|f-\frac{P_\omega(fk_{z})}{k_{z}} \right|^2\omega\,dA\nonumber\\
			&\asymp \fint_{D(z,\delta \tau_\omega(z))} \left|f-\frac{P_\omega(fk_{z})}{k_{z}} \right|^2\,dA\nonumber\\
			&\geq G_{\delta,\omega}(f)(z)^2.
		\end{align}
Since the left-hand side is bounded by $\|H_f\|^2$, this shows that $f\in BDA_\omega(\Omega)$.

$(\Longleftarrow)$ Suppose that $f\in BDA_\omega(\Omega)$. To show the required assertion, we first apply Lemma \ref{decomposelemma}, with $\delta$ being replaced by $\delta/B_\omega^2$, to write $f=f_1+f_2$ such that $f_1\in C^\infty(\Omega)$ and
\begin{align}\label{keyinequality0}
\tau_\omega(z)\bar{\partial}f_1(z)+M_{\delta/B_\omega^2,\omega} (\tau_\omega \bar{\partial}f_1)(z)+M_{\delta/B_\omega^2,\omega} (f_2)(z)\lesssim G_{\delta,\omega}(f)(z)
\end{align}
for all $z\in \Omega.$ It suffices to show that $H_{f_1}$ and $H_{f_2}$ are well-defined and both bounded. Set $d\mu:=\tau_\omega^2|\bar{\partial}f_1|^2dA$ and $d\nu:=|f_2|^2\,dA$. Combining inequality \eqref{keyinequality0} with \cite[Theorem 5.1]{MR3490780},
we deduce that Toeplitz operators $T_\mu$ and $T_\nu$ are bounded on $A_\omega^2$, with operator norms dominated by
$\displaystyle \|f\|_{BDA_\omega(\Omega)}^2$.
Equivalently, the embedding operators $J_{\mu}$ (from $A^2_\omega$ to $L_\omega^2(\mu)$) and $J_{\nu}$ (from $A^2_\omega$ to  $L_\omega^2(\nu)$) satisfy the inequality
		\begin{align}\label{carlesonin}
\|J_\mu\|^2+\|J_\nu\|^2\lesssim \displaystyle \|f\|_{BDA_\omega(\Omega)}^2.
\end{align}
since $T_\mu=J_\mu^* J_\mu$ and $T_\nu=J_\nu^* J_\nu$.  It follows from inequality \eqref{carlesonin} that for all $g\in \Gamma$, we have
		\begin{align}\label{rec2}
	\|f_2g\|_{L_\omega^2}^2=\|J_\nu g\|_{L_\omega^2(\nu )}^2\lesssim \|f\|_{BDA_\omega(\Omega)}^2
		\|g\|_{A_\omega^2}^2.
		\end{align}
This implies that $f_2\in \mathcal{S}$ (and thus, $f_1\in \mathcal{S}$). Therefore, $H_{f_1}$ and $H_{f_2}$ are well-defined. Moreover, since $\|H_{f_2}g\|_{L_\omega^2}\leq \|f_2g\|_{L_\omega^2}$, we also deduce that $\|H_{f_2}\|\lesssim \|f\|_{BDA_\omega(\Omega)}.$

We now turn to $H_{f_1}$. Applying Corollary \ref{keycoro} and then  inequality \eqref{carlesonin}, we conclude that
		\begin{align}\label{rec1}
			\|H_{f_1}g\|_{L_\omega^2}^2&\lesssim \int_\Omega \tau_\omega(z)^2|\bar{\partial}f_1(z)|^2|g(z)|^2 dA_\omega(z)=\|J_\mu g\|_{L_\omega^2(\mu)}^2
			\lesssim \|f\|_{BDA_\omega(\Omega)}^2
			\|g\|_{A_\omega^2}^2.
		\end{align}
		This ends the proof of Theorem \ref{boundedness}.
	\end{proof}

\begin{theorem}\label{compactness}
Let $\omega \in \mathcal{W}^\ast(\Omega)$, $\delta\in (0,\delta _\omega)$ and $f\in BDA_\omega(\Omega)$, then the Hankel operator $H_f$ is compact from $A^2_\omega$ to $L_\omega^2$
	if and only if $f\in VDA_\omega(\Omega)$.
\end{theorem}
\begin{proof}
$(\Longrightarrow)$
Suppose that $H_f$ is compact from $A^2_\omega$ to $L_\omega^2$. Recall from inequality \eqref{uselater} that
		$$G_{\delta,\omega}(f)(z)\lesssim \|H_f k_{z}\|_{L_\omega^2}.$$
Since $\omega\in\mathcal{W}^*(\Omega)$, $k_z$ converges weakly to $0$ as $z\rightarrow \partial_\infty\Omega$. Combining the above two facts, we deduce that
		$$\lim_{z\rightarrow \partial_\infty\Omega}G_{\delta,\omega}(f)(z)\lesssim \lim_{z\rightarrow \partial_\infty\Omega}\|H_f k_{z}\|_{L_\omega^2}=0.$$
This implies that $f\in VDA_\omega(\Omega)$.

$(\Longleftarrow)$ Suppose that $f\in VDA_\omega(\Omega)$. We decompose $f=f_1+f_2$ as in the proof of Theorem \ref{boundedness}. Set $d\mu:=\tau_\omega^2|\bar{\partial}f_1|^2dA$ and $d\nu=|f_2|^2\,dA$. It follows from \cite[Theorem 5.2]{MR3490780} and inequality \eqref{keyinequality0} that $J_\mu$ and $J_\nu$ are vanishing Carleson measures. Recall from inequalities \eqref{rec1} and \eqref{rec2} that
		$$\|H_{f_1}g\|_{L_\omega^2}\lesssim \|J_\mu g\|_{L_\omega^2(\mu)}\quad {\rm and}\quad \|H_{f_2}g\|_{L_\omega^2}\lesssim \|J_\nu g\|_{L_\omega^2(\mu)}.$$
Hence, $H_{f_1}$ and $H_{f_2}$ are compact from $A^2_\omega$ to $L_\omega^2$. Furthermore, $H_{f}$ is compact from $A^2_\omega$ to $L_\omega^2$. This finishes the proof of Theorem \ref{compactness}.
	\end{proof}
\section{Asymptotic behavior of singular values}\label{mainsection}
This section is devoted to presenting the proof of Theorems \ref{main0}, \ref{simultaneous} and \ref{fastdecay}. As a byproduct, the Schatten-$p$ class characterization of Hankel operators for $1<p<\infty$ in the $W^*(\Omega)$-weighted setting are also obtained.
\subsection{Singular values with slow decay}
The following lemma will be used in the proof of the lower estimate of $\text{Tr}\ h(|  H_f|)$.
\begin{lemma}\label{lower estimatekey}
Let $\omega \in \mathcal{W} ^\ast(\Omega)$ and $f\in VDA_\omega(\Omega)$, then for any increasing convex function $h$ such that $h(0)=0$, we have
$$
\displaystyle \sum _{n\geq 0}h(\|  \chi _{B_\omega R_n}H_fk_{z_n} \|_{L_\omega^2}) \lesssim \displaystyle \sum _{n\geq 0}h(s_n(H_f )),
$$
where the implicit constant depends only on $\{B_\omega R_n\}_n$ and $\omega$.
\end{lemma}
\begin{proof}
We refer the reader to \cite[Lemma 4.2]{MR4413302} for the proof. The statement there is stated for anti-analytic symbol, but the proof also applies to general symbol $f\in VDA_\omega(\Omega)$.
%
\end{proof}

\begin{theorem}\label{main}
Let $\omega \in \mathcal{W}^\ast(\Omega)$, $\delta\in (0,\delta _\omega)$ and $f\in VDA_\omega(\Omega)$. Let $h: [0,+\infty)\to  [0,+\infty)$ be an increasing convex  function  such that $h(0)=0$, then there exist constants $B,C>0$, which depends only on $\delta$ and $\omega$, such that
	$$
	\displaystyle \frac{1}{C}\int _\Omega h\left ( \frac{1}{B}G_{\delta,\omega}(f)(z)\right ) d\lambda _\omega (z) \leq  \|h(|H_{f}|)\|_{S^1} \leq \displaystyle C\int _\Omega h\left ( BG_{\delta,\omega}(f)(z)\right ) d\lambda _\omega (z).
	$$
\end{theorem}

\begin{proof}
By spectral decomposition theorem, we have
$$\|h(|H_{f}|)\|_{S^1}=\displaystyle \sum _{n\geq 0}h\big(s _n(H_{f} ) \big ). $$
Now we divide our proof into two steps.

{\bf Step 1.} We prove that for any $\delta\in (0,\delta _\omega)$, there exist constants $B,C>0$ such that $$ \displaystyle \sum _{n\geq 0}h\big (s _n(H_{f} ) \big ) \leq  \displaystyle C\int_{\Omega} h(BG_{\delta,\omega}(f)(z))\,d\lambda_\omega(z).$$
To begin with, we apply Lemma \ref{decomposelemma}, with $\delta$ being replaced by $\delta/B_\omega^3$, to write $f=f_1+f_2$ such that $f_1\in C^\infty(\Omega)$ and
\begin{align}\label{keyinequality}
\tau_\omega(z)\bar{\partial}f_1(z)+M_{\delta/B_\omega^3,\omega} (\tau_\omega \bar{\partial}f_1)(z)+M_{\delta/B_\omega^3,\omega} (f_2)(z)\lesssim G_{\delta/B_\omega,\omega}(f)(z)
\end{align}
for all $z\in \Omega.$ Since $f\in VDA_\omega(\Omega)\subset BDA_\omega(\Omega)$, it can be seen from the proof of Theorem \ref{boundedness} that $f_1,f_2\in\mathcal{S}$.

By Corollary \ref{keycoro}, for any $g\in \Gamma$,
	\begin{equation}\label{d-bar2}
	\int_\Omega |H_{f_1}g(z)|^2dA_\omega(z) \lesssim \int_\Omega \tau_\omega(z)^2|\bar{\partial}f_1(z)|^2|g(z)|^2 dA_\omega(z).
	\end{equation}
The equation $(\ref{d-bar2})$ implies that
	\begin{equation} \label{hankel toeplitz}
	H_{f_1}^\ast H_{f_1}\lesssim T_{\mu_{f_1}},
	\end{equation}
where $d\mu_{f_1}:=\tau_\omega^2|\bar{\partial}f_1|^2dA$. Moreover, we have
	\begin{equation}\label{obvious}
\|H_{f_2}g\|_{L_\omega^2}^2=\|(I-P_\omega)(f_2g)\|_{L_\omega^2}^2\leq \|M_{|f_2|}g\|_{L_\omega^2}^2.
\end{equation}
The equation $\eqref{obvious}$ implies that
	\begin{equation} \label{hankel toeplitz}
	H_{f_2}^\ast H_{f_2}\lesssim T_{\mu_{f_2}},
	\end{equation}
where $d\mu_{f_2}:=|f_2|^2dA$.
	By Lemma \ref{Weyl's Lemma}, for $j=1,2$,
\begin{equation} \label{HT}
	s_n^2(H_{f_j})=\lambda_{n}(H_{f_j}^\ast H_{f_j}) \lesssim   \lambda _n(T_{\mu_{f_j}}).
	\end{equation}
Let $\tilde{h}(t):= h(\sqrt{t})$.
By \cite[Theorem 4.5]{MR4516170}, for $j=1,2$,
		\begin{align}\label{subtt}
			\sum_{n\geq 0} h(s_n(H_{f_j})) & = \sum_{n\geq 0} \tilde{h}(s_n^2(H_{f_j}))
			\leq \sum_{n\geq 0} \tilde{h}( B\lambda _n(T_{\mu _{f_j}}))
			\leq \sum_{n\geq 0}  \tilde{h}\left(B \frac{\mu _{f_j}(R_n)}{A(R_n)}\right),
		\end{align}
where in the last inequality we choose $\{R_n\}_n= \{ D (z_n, \frac{\delta}{B_\omega^3} \tau _\omega (z_n))\}_n\in  \mathcal{L}_\omega (\frac{\delta}{B_\omega^3} )$.
By inequality \eqref{keyinequality},
\begin{align}\label{substi1}
\sum_{n\geq 0}  \tilde{h}\left(B \frac{\mu _{f_1}(R_n)}{A(R_n)}\right)&=  \sum_{n\geq 0}  \tilde{h} \left( B \fint_{R_n} |\tau_\omega(z)|^2|\bar{\partial}f_1(z)|^2 dA(z)\right)\nonumber\\
			&=\sum_{n\geq 0} h(BM_{\delta/B_\omega^3,\omega}(\tau_\omega|\bar{\partial}f_1|)(z_n))\nonumber\\
			&\leq \sum_{n\geq 0} h(BG_{\delta/B_\omega,\omega}(f)(z_n)).
\end{align}
Similarly, applying inequality \eqref{keyinequality}  again, we deduce that
		\begin{align}\label{substi2}
			\sum_{n\geq 0}  \tilde{h}\left(B \frac{\mu _{f_2}(R_n)}{A(R_n)}\right) \leq \sum_{n\geq 0} h(BG_{\delta/B_\omega,\omega}(f)(z_n)).
		\end{align}
Substituting inequalities \eqref{substi1} and \eqref{substi2} into inequality \eqref{subtt}, we conclude that for $j=1,2$,
\begin{align*}
\sum_{n\geq 0} h(s_n(H_{f_j}))\leq \sum_{n\geq 0} h(BG_{\delta/B_\omega,\omega}(f)(z_n)).
\end{align*}
Applying \cite[Corollary 1.35]{MR2311536}, we see that for $j\in\{0,1\}$ and for any nonnegative integer $k$,
\begin{align}\label{plislem}
s_{2k+j}(H_{f})\leq s_{k+j}(H_{f_1})+s_{k+1}(H_{f_2}).
\end{align}
		From this and the fact that $h$ is increasing and convex, we obtain
		\begin{align*}
			\sum_{n\geq 0} h(s_n(H_{f}))
			&\lesssim \sum_{n\geq0} \left(h(2s_n(H_{f_1}))+h(2s_n(H_{f_2}))\right)\\
			&\lesssim \sum_{n\geq0} h(BG_{\delta/B_\omega,\omega}(f)(z_n)).
		\end{align*}
Since $G_{\delta/B_\omega,\omega}(f)(z_n)\lesssim G_{\delta,\omega}(f)(z)$ for any $z\in R_n$, we have
		\begin{align*}
			h(BG_{\delta/B_\omega,\omega}(f)(z_n))&\leq h\left(B\fint_{R_n}G_{\delta,\omega}(f)(z)\,dA(z)\right)\\
			&\leq \int_{R_n}h(BG_{\delta,\omega}(f)(z))\frac{dA(z)}{A(R_n)}\\
			&\lesssim \int_{R_n}h(BG_{\delta,\omega}(f)(z))\,d\lambda_\omega(z).
		\end{align*}
		Therefore, by the finite multiplicity of $(R_n)_n$, we get
		$$\sum _{n\geq 0}h\left (s _n(H_{f} ) \right )\lesssim  \int_{\Omega} h(BG_{\delta,\omega}(f)(z))\,d\lambda_\omega(z).$$

{\bf Step 2.} We prove that for any $\delta\in (0,\delta _\omega)$, there exist constants $B,C>0$ such that $$ \displaystyle \frac1C\int_{\Omega} h(\frac{1}{B}G_{\delta,\omega}(f)(z))\,d\lambda_\omega(z)\leq \sum _{n\geq 0}h\big (s _n(H_{f} ) \big ) \displaystyle .$$
To begin with, by Lemma \ref{lower estimatekey}, we have
		\begin{align}\label{combb1}
\sum_{n\geq 0} h(\|\chi_{B_\omega R_n}H_f k_{z_n}\|_{L_\omega^2})\lesssim \sum_{n\geq0} h(s_n(H_{f})).
\end{align}
Note that for any $\xi \in R_n$,
		$$|k_{z_n}(\xi)|\asymp \|K_\xi\|_{A_\omega^2}= \tau_\omega(\xi)^{-1}\omega(\xi)^{-1/2}.$$
		Therefore, $k_{z_n}^{-1}\in \text{Hol}(R_n)$ and
		\begin{align*}
			\|\chi_{B_\omega R_n}H_f k_{z_n}\|_{L_\omega^2}^2&=\int_{B_\omega R_n} |fk_{z_n}-P_\omega(fk_{z_n})|^2\omega\,dA\nonumber\\
			&=\int_{B_\omega R_n} |k_{z_n}|^2\left|f-\frac{P_\omega(fk_{z_n})}{k_{z_n}} \right|^2\omega\,dA\nonumber\\
			&\asymp \fint_{B_\omega R_n} \left|f-\frac{P_\omega(fk_{z_n})}{k_{z_n}} \right|^2\,dA\nonumber\\
			&\gtrsim G_{B_\omega \delta,\omega}(f)(z_n)^2.
		\end{align*}
Combining this inequality with the facts that $G_{\delta,\omega}(f)(z)\lesssim  G_{B_\omega\delta,\omega}(f)(z_n)$ for any $z\in R_n$ and that $h$ is an increasing function, we conclude that there exists a constant $B>0$ such that
\begin{align}\label{combb3}
\int_{R_n} h(\frac{1}{B}G_{\delta,\omega}(f)(z)) \,d\lambda_\omega(z)\lesssim  h( \|\chi_{B_\omega R_n}H_f k_{z_n}\|_{L_\omega^2}).
\end{align}
Combining inequalities \eqref{combb1} and \eqref{combb3} together, we conclude that
		$$\int_{\Omega} h(\frac{1}{B}G_{\delta,\omega}(f)(z))\,d\lambda_\omega(z)\lesssim \sum_{n\geq0} h(s_n(H_{f})).$$
This completes the proof of Theorem \ref{main}.
\end{proof}

\begin{coro}\label{SpandSpinfty}
Let $\omega \in \mathcal{W}^\ast(\Omega)$, $\delta\in (0,\delta _\omega)$ and $f\in \mathcal{S}\cap VDA_\omega(\Omega)$, then for any $1\leq p<+\infty$, $H_f\in S^p(A^2_\omega \to L_\omega^2)$
	if and only if $f\in IDA_\omega^p(\Omega)$.
\end{coro}
\begin{proof}
The proof is a direct consequence of Theorem \ref{main} by choosing $h(t)=t^p$.
\end{proof}
\begin{proof}[Proof of Theorem \ref{main} \eqref{AAAAA}]
We first apply \cite[Proposition 1.1.4 and Proposition 1.4.5 (12)]{MR3243734} to see that for every increasing function $h: [0,+\infty)\to  [0,+\infty)$ satisfying $h(0)=0$,
\begin{align}\label{disformm}
\int_{\Omega} h(G_{\delta,\omega}(f)(z))\,d\lambda_\omega(z)=\int_0^\infty h((G_{\delta,\omega}(f))^*(t))dt.
\end{align}
This, in combination with Theorem \ref{main}, yields
\begin{align*}
\int_0^\infty h((G_{\delta,\omega}(f))^*(t))dt\lesssim \displaystyle \sum _{n\geq 0}h\big (s _n(H_{f} ) \big ) \lesssim \int_0^\infty h((G_{\delta,\omega}(f))^*(t))dt.
\end{align*}
This, together with Lemma \ref{Convex1}, yields
$$
s_n(H_{f} ) \lesssim 1/ \rho (n),\ {\rm for}\ {\rm all}\ n\in\mathbb{N}  \Longleftrightarrow  (G_{\delta,\omega}(f))^*(n)  \lesssim 1/ \rho (n),\ {\rm for}\ {\rm all}\ n\in\mathbb{N}.
$$
This completes the proof of Theorem \ref{main} \eqref{AAAAA}.
\end{proof}

\begin{lemma}\label{simmu}
Let $\omega \in \mathcal{W}^\ast(\Omega)$, $\delta\in (0,\delta _\omega)$  and $f\in L_{\rm loc}^2(\Omega)$, then
\begin{align*}
 G_{\delta,\omega}(f)(z)+G_{\delta,\omega}(\bar {f})(z)\asymp MO_{\delta,\omega} (f)(z),\quad {\rm for}\ {\rm all}\ z\in\Omega.
\end{align*}
\end{lemma}
\begin{proof}
It follows from the definition that $G_{\delta,\omega}(f)(z)\lesssim MO_{\delta,\omega} (f)(z)$ and $G_{\delta,\omega}(\bar{f})(z)\lesssim MO_{\delta,\omega} (\bar{f})(z)=MO_{\delta,\omega} (f)(z)$ for all $z\in\Omega$. Conversely,
as in the proof of \cite[Lemma 6.1]{MR4402674}, we have
\begin{align*}
MO_{\delta,\omega} (f)(z)\lesssim G_{\delta,\omega}(f)(z)+G_{\delta,\omega}(\bar {f})(z),\quad {\rm for}\ {\rm all}\ z\in\Omega.
\end{align*}
This yields the required conclusion.
\end{proof}

\begin{proof}[Proof of Theorem \ref{simultaneous}]
The simultaneous membership of $s_n(H_f)$ and $s_n(H_{\bar{f}})$ follows directly from Theorem \ref{main}, in combination with Lemma \ref{simmu} and \cite[Proposition 1.4.5 (4)]{MR3243734}.
\end{proof}

\subsection{Singular values with fast dacay}
Inspired by the argument in \cite[Theorem 4.5]{MR4413302}, we now show that the decay assumption of $\rho$ on one side of Theorem \ref{main} \eqref{AAAAA} can be weaken.
\begin{proof}[Proof of Theorem \ref{main} \eqref{BBBBB}]
We decompose $f$ as in \eqref{keyinequality}. Following the proof of Theorem \ref{main}, we see that for any increasing function $h:[0,\infty)\to [0,\infty)$ such that $h(t^p)$ is convex for some $p>\gamma$,
		\begin{align*}
			\sum_{n\geq 0} h(s_{n}(H_{f_j}))\leq \sum_{n\geq 0} h(BG_{\delta/B_\omega,\omega}(f)(z_n)).
		\end{align*}
		Since $G_{\delta/B_\omega,\omega}(f)(z_n)\lesssim G_{\delta,\omega}(f)(z)$ for any $z\in R_n$, by the convexity of $h(t^p)$, there is a constant $\tilde{B}>0$ such that
		\begin{align*}
			h(BG_{\delta/B_\omega,\omega}(f)(z_n))&\leq h\left(\left[\fint_{R_n}\tilde{B}G^{1/p}_{\delta,\omega}(f)(z)\,dA(z)\right]^p\right)\\
			&\leq \int_{R_n}h(\tilde{B}G_{\delta,\omega}(f)(z))\frac{dA(z)}{A(R_n)}\\
			&\lesssim \int_{R_n}h(\tilde{B}G_{\delta,\omega}(f)(z))\,d\lambda_\omega(z).
		\end{align*}
		By the finite multiplicity of $(R_n)_n$, we have
		$$\sum _{n\geq 0}h\left (s_n(H_{f_j} ) \right )\lesssim  \int_{\Omega} h(\tilde{B}G_{\delta,\omega}(f)(z))\,d\lambda_\omega(z).$$
This, in combination with formula \eqref{disformm} and Lemma \ref{finelemma},  implies that for $j\in\{0,1\}$,
\begin{align}\label{nboopp}
s_n(H_{f_j})\lesssim 1/\rho(n).
\end{align}
Using inequalities \eqref{plislem} and \eqref{nboopp} and then applying the monotonicity assumption of $\rho$, we conclude that $s_n(H_f)\lesssim 1/\rho(n)$ for all $n\in\mathbb{N}$.
\end{proof}

In the rest of this subsection, we turn to the proof of Theorem \ref{fastdecay}. For convenience, we will use the following concrete $\delta$-lattice $\Lambda$ in $\mathbb{C}$ instead of the lattice given in Lemma \ref{lattice}. Given $\delta>0$ and $w_1\in \mathbb{C}$, we denote
$
	\Lambda:= \left \{w_1+ \delta  ( m+ \textrm i s): m, s  \in {\mathbb Z} \right \}.
$
For  fixed $K\in \mathbb N$, we write
$$
	\left\{ w_1+ z\in \Lambda:  0
	\le  \mathrm {Re}\, z, \mathrm {Im}\, z <K \delta \right\}
	=: \left\{w_1, \ldots, w_{K^{2}}\right\},
$$
and for $1\le k\le K^{2}$,
$$
	\Lambda_k:=  \left \{w_k+ K\delta  ( m+ \textrm i s): m, s \in {\mathbb Z} \right \}.
$$
Then
\begin{align}\label{newlattice}
	\Lambda= \bigcup_{k=1}^{K^{2}} \Lambda_k
\end{align}
 with $\Lambda_j\cap \Lambda_k
	= \emptyset$ if $j\neq k$, and with
$
 |a-b|\ge K\delta  \ \textrm {if}\ a, b \in \Lambda_k\ {\rm and}\ a\ne b.
$


\begin{theorem}
Let $f\in \mathcal S$ and suppose that $\varphi\in C^2(\mathbb{C})$ is real-valued with $\mathrm{i} \partial \bar{\partial} \varphi \simeq \omega_0$, where $\omega_0=i\partial\bar{\partial}|z|^2$ is the Euclidean-K\"{a}hler form on $\mathbb{C}$. Then for any $0<p<\infty$, $H_f\in S^{p,\infty}(F^2(\varphi) \to L^2(\varphi))$ if and only if $f\in \mathrm{IDA}^{p,\infty}(\mathbb{C})$.
\end{theorem}
\begin{proof}
By Theorem \ref{main} \eqref{BBBBB}, it remains to show that $H_f\in S^{p,\infty}(F^2(\varphi) \to L^2(\varphi))$ implies $f\in \mathrm{IDA}^{p,\infty}(\mathbb{C})$, where $0<p\leq1$. We will achieve this by borrowing an idea from the proof of \cite[Theorem 1.1]{MR4402674} (see also \cite[Proposition~6.8]{MR3770363}). To begin with, recall from \cite[Lemma 2.2]{MR4402674} that there exists positive constants $C$ and $\delta_0 $ such that the reproducing kernel $K$ of $F^2(\varphi)$ satisfies
\begin{equation}\label{basic-est-b}
	\left|K(z,w)\right|\geq C e^{ \frac{1}{2}(\varphi(z)+\varphi(w))}
\end{equation}
for $z\in\mathbb{C}$ and  $w\in D\left(z, \delta_0\right)$.
Now we let $\delta\in (0, \delta_0)$, and let $\Lambda$ be a $\delta$-lattice, and decompose $\Lambda= \bigcup_{k=1}^{K^{2}} \Lambda_k$ as in \eqref{newlattice} with $K\ge 2$  such that $D(a , \delta)\cap D(b , \delta)= \emptyset$ if $a\neq b$ and
$a, b\in \Lambda _k$.

For simplicity, we denote by $P$ the orthogonal projection of $L^2(\varphi)$ onto $F^2(\varphi)$. For $a\in \mathbb{C}$ and  $\delta>0$, let  $A^2(D(a, \delta), e^{-\varphi} dA)$ be the weighted Bergman space of all holomorphic functions in the space $L^2(D(a, \delta), e^{-\varphi} dA)$. We denote by $P_{a, \delta}$ the orthogonal projection of $L^2(D(a, \delta), e^{-\varphi} dA)$ onto $A^2(D(a,\delta), e^{-\varphi} dA)$. Given $f\in L^2(D(a, \delta), e^{-\varphi} dA)$, we extend $P_{a, \delta}(f)$ to $\mathbb{C}$ by setting
$
	P_{a, \delta}(f)|_{\mathbb{C}\backslash D(a, \delta)} =0.
$
It follows from \cite[Theorem 4.4]{MR4402674} that $fk_a-P(fk_a) \in L^2_{\mathrm{loc}}(\mathbb{C})$. This, together with a trivial fact that $P(fk_a) \in {\rm Hol}(\mathbb{C})$, implies that
 $
 fk_a\in L^2(D(a, \delta), e^{-\varphi}dA)$ and $P_{a, r}(fk_a)\in A^2(D(a,\delta), e^{-\varphi} dA).
 $
 Hence, we have $\chi_{D(a, \delta)} fk_{a}-P_{a, \delta} (fk_{a})\in L^2(\varphi)$.  Next, for $a\in \Lambda_k$, we define
\begin{equation*}
g_a:= \begin{cases}
	\frac {\chi_{D(a , \delta)} fk_{a}-P_{a, \delta} (fk_{a })}{\|\chi_{D(a, \delta)} fk_{a}-P_{a, \delta} (fk_{a})\|_{L^2(\varphi)}},
	& \textrm { if } \|\chi_{D(a, \delta)} fk_{a}-P_{a, \delta} (fk_{a})\|_{L^2(\varphi)}\neq 0, \\
	0,  & \textrm { if } \|\chi_{D(a, \delta)} fk_{a}-P_{a, \delta} (fk_{a})\|_{L^2(\varphi)}= 0.
\end{cases}
\end{equation*}
Then the sequence $\{g_a\}_{a\in\Lambda_k}$ satisfies that $\|g_a\|_{L^2(\varphi)}\le 1$ and $\langle g_{a}, g_{b} \rangle_{L^2(\varphi)}=0 $ if $D(a ,\delta)\cap D(b , \delta)= \emptyset$. Let $\Theta$ be any finite sub-collection of $\Lambda_k$,
and let $\{e_a\}_{a\in \Theta} $ be an orthonormal set of $L^2(\varphi)$. Let
$$
	A_\Theta:= \sum_{a\in \Theta} e_a\otimes g_a\quad {\rm and}\quad T_\Theta:= \sum_{a\in \Theta}  k_{a}\otimes e_a,
$$
where for any $U, V\in L^2(\varphi)$, $U\otimes V$ is a rank one operator on $L^2(\varphi)$ defined by
$$
	(U\otimes V)(g):= \langle g, V \rangle_{L^2(\varphi)} U,  \quad g\in L^2(\varphi).
$$
Then $A_\Theta$ is a finite-rank bounded operator on $L^2(\varphi)$ with operator norm bounded by $1$, and by \cite[Lemma 2.4]{MR4402674}, $T_\Theta$ is bounded from $L^2(\varphi)$ to $F^2(\varphi)$ since  $\Lambda$ is separated. Note that
\begin{equation}\label{AHT-a}
 	A_\Theta H_fT_\Theta= \sum_{a, \tau \in \Theta} \left \langle  H_f k_{ \tau}, g_a \right \rangle_{L^2(\varphi)} e_a\otimes e_\tau
	=:Y_\Theta+Z_\Theta,
\end{equation}
where
\begin{equation}\label{AHT}
	Y_\Theta:=  \sum_{a\in \Theta } \left \langle  H_f k_{a}, g_a \right\rangle_{L^2(\varphi)} e_a\otimes e_a,\ \
	Z_\Theta:= \sum_{a, \tau \in \Theta, \, a\neq\tau} \left \langle  H_f k_{ \tau}, g_a\right \rangle_{L^2(\varphi)} e_a\otimes e_\tau.
\end{equation}
By \cite[Lemma 5.1]{MR4402674} and inequality \eqref{basic-est-b},
\begin{align*}
	\left  \langle  H_f k_{a }, g_a \right \rangle_{L^2(\varphi)}
	&=\left \langle   f k_{a }-P(f k_{a}),   g_a \right \rangle_{L^2(\varphi)}\\
	&=\left \langle \chi_{D(a  , \delta)}   f k_{a }-P_{a , \delta}(f k_{a }), g_a \right \rangle_{L^2(\varphi)}\\
	&=\left \|  \chi_{D(a, \delta)}   f k_{a }-P_{a ,\delta}(f k_{a }) \right\|_{L^2(\varphi)}\\
	&\gtrsim\left \|   f-\frac 1 {  k_{a } }P_{a , \delta}(f k_{a }) \right\|_{L^2(D(a , \delta), dA)}\\
&\gtrsim G_{\delta}(f)( a ).
\end{align*}
Therefore, there is a constant $\kappa$ independent of $f$ and $\Theta$ such that
\begin{equation}\label{Y-est-a}
	\|Y_\Theta\|_{S^{p,\infty}}=\left\|\left\{\left  \langle  H_f k_{a }, g_a \right \rangle_{L^2(\varphi)}\right\}_{a\in\Theta}\right\|_{\ell^{p,\infty}}
	\geq   \kappa \|\{G_{\delta}(f)( a )\}_{a\in \Theta}\|_{\ell^{p,\infty}}.
\end{equation}
To handle the off-diagonal operator $Z_\Theta$, we first consider the linear operator $G_\Theta$ given by
$$G_\Theta:\{\alpha_{a,\tau}\}_{a,\tau\in\Theta}\rightarrow \sum_{a,\tau\in\Theta}\alpha_{a,\tau}e_a\otimes e_{\tau}.$$
Then we apply \cite[Proposition 1.29]{MR2311536} to see that $G_\Theta$ is bounded from $\ell^p$ to $S^p$ with operator norm bounded by $1$. This, in combination with real interpolation (see \cite[Theorems 3.11.8 and 5.2.1]{Bergh}), implies that $G_\Theta$ is bounded from $\ell^{p,\infty}$ to $S^{p,\infty}$ with operator norm bounded by $1$. In particular,
\begin{align}\label{Z-est}
	\left\|Z_\Theta\right\|_{S^{p,\infty}}
	\leq   \left\|\left\{\left \langle  H_f k_{ \tau}, g_a\right \rangle_{L^2(\varphi)}\right\}_{a, \tau \in \Theta}\right\|_{\ell^{p,\infty}}.
\end{align}
To continue, we let $Q_{a, \delta}$ be the Bergman projection of $L^2(D(a, \delta), dA)$ onto the Bergman space $A^2(D(a, \delta), dA)$. Then
$$
	k_\tau Q_{a, \delta}(f)\in A^2(D(a, \delta), dA)= A^2(D(a,\delta), e^{-\varphi} dA).
$$
Moreover, $fk_\tau -P_{a, \delta}(fk_\tau)$ and $P_{a,\delta}(fk_\tau)- k_\tau Q_{a, \delta}f $ are orthogonal in $L^2(D(a, \delta), e^{-\varphi} dA)$.  Hence, we apply Parseval's identity to deduce that for $a, \tau\in \mathbb{C}$,
$$
	\left\|  fk_\tau -P_{a, \delta}(fk_\tau)\right\|_{L^2(D(a, \delta), e^{-\varphi} dA)}
	\le \left\|  fk_\tau -k_\tau Q_{a, \delta}f\right\|_{L^2(D(a, \delta), e^{-\varphi} dA)}.
$$
This, in combination with \cite[Lemma 5.1]{MR4402674} and \cite[(2.4)]{MR4402674}, yields
\begin{align*}
	\left| \left \langle  H_f k_{ \tau}, g_a \right\rangle_{L^2(\varphi)} \right|
	&=\left| \left \langle   f k_{\tau} -P \left( f k_{ \tau}  \right), g_a  \right\rangle_{L^2(\varphi)} \right|\\
	&=\left| \left \langle  \chi_{D(a, \delta)} f k_{ \tau} -P_{a , \delta} \left( f k_{ \tau}\right), g_a \right\rangle_{L^2(\varphi)} \right|\\
	&\le\left\|  f k_{ \tau}-P_{a,\delta}  \left( f k_{ \tau} \right) \right\|_{L^2(D(a, \delta), e^{-\varphi} dA)}\\
	&\le\left\|  f k_{ \tau}- k_{ \tau} Q_{a, \delta}  \left(f\right) \right\|_{L^2(D(a, \delta), e^{-\varphi} dA)}\\
	&\le\sup_{\xi\in D(a, \delta)} \left|k_{ \tau}(\xi)e^{-\frac{\varphi}{2}}\right |\, \left\|  f-  Q_{a , \delta}  \left( f  \right) \right\|_{L^2(D(a, \delta), dA)}\\
	&\le C e^{-\theta | a -  \tau|} \left\| f- Q_{a , \delta}  \left( f  \right) \right\|_{L^2(  D(a, \delta), dA)}\\
&=Ce^{-\theta | a -  \tau|}G_{\delta}(f)(a).
\end{align*}
Therefore, we have
\begin{align}\label{rhssofff}
\left\|\left\{\left \langle  H_f k_{ \tau}, g_a\right \rangle_{L^2(\varphi)}\right\}_{a, \tau \in \Theta}\right\|_{\ell^{p,\infty}}
&\lesssim \left\|\left\{e^{-\theta | a -  \tau|}G_{\delta}(f)(a)\right\}_{a, \tau \in \Theta}\right\|_{\ell^{p,\infty}}\nonumber\\
&\lesssim e^{-\frac{1}{2}\theta \delta K}\left\|\left\{e^{-\frac{\theta}{2} | a -  \tau|}G_{\delta}(f)(a)\right\}_{a, \tau \in \Theta}\right\|_{\ell^{p,\infty}}.
\end{align}
We will apply again a real interpolation argument to estimate the $\ell^{p,\infty}$ norm of $\left\{e^{-\frac{\theta}{2} | a -  \tau|}G_{\delta}(f)(a)\right\}_{a, \tau \in \Theta}$. To this end, given a sequence $\gamma:=\{\gamma_{a,\tau}\}_{a,\tau\in\Theta}$, we consider the linear operator $H_{\Theta,\gamma}$ given by
\begin{align*}
H_{\Theta,\gamma}:\{\beta_{a}\}_{a\in\Theta}\rightarrow \{\beta_{a}\gamma_{a,\tau}\}_{a,\tau\in\Theta}.
\end{align*}
Note that $H_{\Theta,\gamma}$ is bounded from $\ell^p(\Theta)$ to $\ell^p(\Theta\times \Theta)$ for any $0<p\leq1$ with operator norm bounded by $\sup\limits_{\tau\in\Theta}\left(\sum_{a\in\Theta}|\gamma_{a,\tau}|^p\right)^{1/p}$. Then by real interpolation (see \cite[Theorems 3.11.8 and 5.2.1]{Bergh}), it is also bounded from $\ell^{p,\infty}(\Theta)$ to $\ell^{p,\infty}(\Theta\times \Theta)$ for any $0<p\leq 1$ with the same operator norm. Using this fact, we see that the right-hand side of \eqref{rhssofff} is dominated by
\begin{align}\label{cb2356}
&Ce^{-\frac{1}{2}\theta \delta K}\sup\limits_{a\in \Theta}\left(\sum_{\tau\in \Theta,\tau\neq \theta}e^{-\frac{p\theta}{2}|a-\tau|}\right)^{1/p}\left\|\left\{G_{\delta}(f)(a)\right\}_{a \in \Theta}\right\|_{\ell^{p,\infty}}\nonumber\\
&\leq Ce^{-\frac{1}{2}\theta \delta K}\sup\limits_{a\in \Theta}\left(\sum_{\tau\in \Theta,\tau\neq \theta}\int_{D(\tau,\delta)}e^{-\frac{p\theta}{2}|a-\xi|}dA(\xi)\right)^{1/p}\left\|\left\{G_{\delta}(f)(a)\right\}_{a \in \Theta}\right\|_{\ell^{p,\infty}}\nonumber\\
&\leq Ce^{-\frac{1}{2}\theta \delta K} \left\|\left\{G_{\delta}(f)(a)\right\}_{a \in \Theta}\right\|_{\ell^{p,\infty}}
\end{align}
for some constant $C>0$, where in the next to the last inequality we used the fact that  $D(  {\tau} , \delta)\cap D( a , \delta)= \emptyset$ for every $a, \tau\in \Theta$ with $a\neq \tau$. Combining inequalities \eqref{Z-est}, \eqref{rhssofff} and \eqref{cb2356}, we conclude that
\begin{align*}
\left\|Z_\Theta\right\|_{S^{p,\infty}}
	\lesssim e^{-\frac{1}{2}\theta \delta K} \left\|\left\{G_{\delta}(f)(a)\right\}_{a \in \Theta}\right\|_{\ell^{p,\infty}}.
\end{align*}
We choose $K$ to be a sufficiently large constant so that
\begin{equation}\label{Z-est-z}
\left \|Z_\Theta\right\|_{S^{p,\infty}}  \le \frac {\kappa} {2^{1+\frac1p}} \left\|\left\{G_{\delta}(f)(a)\right\}_{a \in \Theta}\right\|_{\ell^{p,\infty}}.
\end{equation}
Combining this inequality with \eqref{Y-est-a} and the quasi-triangular inequality
$$
	\left\| Y_\Theta\right \|_{S^{p,\infty}}\leq 2^{1/p} \left \| A_\Theta H_fT_\Theta\right\|_{S^{p,\infty}}+  2^{1/p}\left\| Z_\Theta\right\|_{S^{p,\infty}},
$$
we deduce that
$$
	\kappa \|\{G_{\delta}(f)( a )\}_{a\in \Theta}\|_{\ell^{p,\infty}}\le 2^{1/p} \left \| A_\Theta H_fT_\Theta\right \|_{S^{p,\infty}} + \frac {\kappa} 2  \|\{G_{\delta}(f)( a )\}_{a\in \Theta}\|_{\ell^{p,\infty}}.
$$
Since $\Theta$  is finite,  we have
\begin{align}\label{AHf-T-norm}
	\|\{G_{\delta}(f)( a )\}_{a\in \Theta}\|_{\ell^{p,\infty}}
\lesssim \left\| A_\Theta H_fT_\Theta\right\|_{S^{p,\infty}}
\lesssim  \| A_\Theta\|_{L^2(\varphi)\to L^2(\varphi)} \| H_f  \|_{S^{p,\infty}}  \|T_\Theta\|_{L^2(\varphi)\to F^2(\varphi)}
	\lesssim\left \| H_f \right \|_{S^{p,\infty}}.
\end{align}
Since the implicit constant above is independent of $f$ and $\Theta$, we have
\begin{align}\label{G-f-est-a}
	\|\{G_{\delta}(f)( a )\}_{a\in \Lambda_k}\|_{\ell^{p,\infty}}
\lesssim \left \| H_f \right \|_{S^{p,\infty}}.
\end{align}
Next, we take $\Lambda$ to be a $\frac \delta {2}$-lattice similar to \eqref{newlattice},  which can be viewed as a union of $4$ $\delta$-lattice. Without loss of generality, we assume that $(G_\delta(f))^*(j)$ is a strictly decreasing sequence (this can be achieved by considering a new strictly decreasing sequence $\tilde{G}_\delta(f)(j)$ such that $\tilde{G}_\delta(f)(j)\asymp (G_\delta(f))^*(j)$).
\begin{align*}
\|G_{\frac{\delta}{2}}(f)\|_{L^{p,\infty}(\mathbb{C})}
&\leq \sup\limits_{\lambda>0}\lambda\left(\sum_{a\in\Lambda}A\left\{z\in D(a,\frac{\delta}{2}):G_{\frac{\delta}{2}}(f)(z)>\lambda\right\}\right)^{1/p}\\
&\lesssim \sup\limits_{\lambda>0}\lambda\left(\sum_{a\in\Lambda}A\left\{z\in D(a,\frac{\delta}{2}):G_{\delta}(f)(a)>\lambda\right\}\right)^{1/p}\\
&=C\sup\limits_{j\geq 0}\sup\limits_{G_\delta(f)^*(j+1)\leq \lambda<G_{\delta}(f)^*(j)}\lambda\left(\#\{a\in\Lambda:G_\delta(f)(a)>\lambda\}\right)^{1/p}\\
&\leq C\sup\limits_{j\geq 0}\sup\limits_{G_\delta(f)^*(j+1)\leq \lambda<G_{\delta}(f)^*(j)}\lambda\left(\#\{a\in\Lambda:G_\delta(f)(a)>G_\delta(f)^*(j+1)\}\right)^{1/p}\\
&\leq C\sup\limits_{j\geq 0}(1+j)^{1/p}G_{\delta}(f)^*(j)=C\|\{G_{\delta}(f)( a )\}_{a\in \Lambda}\|_{\ell^{p,\infty}}\lesssim \left \| H_f \right \|_{S^{p,\infty}}.
\end{align*}
This implies that $f\in IDA^{p,\infty}(\mathbb{C})$.
\end{proof}
\section{Relationship between two characterization functions}\label{relationsection}
Hankel operators associated with anti-analytic symbols are of independent interest in the theory of analytic function (see e.g. \cite{MR970119,MR850538,MR4413302,MR3010276,MR2541276,MR3246989,MR2040916}). The characterizations of their boundedness, compactness, Schatten-$p$ membership and the asymptotic behavior of their singular values are closely linked to the characterization function $\tau_\omega|f'|$. Firstly, this characterization function arises naturally in the definition of the (weighted) Bloch space $\mathcal{B}$, (weighted) little Bloch space $\mathcal{B}_0$ and (weighted) analytic Besov space $\mathcal{B}_p$. The definitions of these spaces are given as follows.
\begin{align*}
&\mathcal{B}^\omega:=\Big\{f\in {\rm Hol}(\Omega):\ \sup_{z\in\Omega}\tau_\omega(z)|f'(z)|<+\infty\Big\},\\
&\mathcal{B}^\omega_0:=\Big\{f\in {\rm Hol}(\Omega):\ \lim_{z\rightarrow \partial_\infty\Omega}\tau_\omega(z)|f'(z)|=0\Big\},\\
&\mathcal{B}^\omega_p:=\Big\{f\in {\rm Hol}(\Omega):\ \tau_\omega|f'|\in L^p(\Omega,d\lambda_\omega)\Big\}.
\end{align*}
Let $f\in {\rm Hol}(\Omega)$. In the setting of standard Bergman space $A_\alpha^2$ where $\omega$ satisfies $\tau_\omega(z)\asymp 1-|z|^2$, J. Arazy, S. Fisher and J. Peetre \cite{MR970119} showed that $H_{\bar{f}}$ is bounded (resp. compact) on $A^2_\alpha$ if and only if $f\in\mathcal{B}^\omega$ (resp. $\mathcal{B}^\omega_0$). These two characterizations were first established by S. Axler \cite{MR850538} in the case $\alpha=0$. The authors in \cite{MR970119} also showed that  $H_{\bar{f}}\in S^p$ if and only if $f\in\mathcal{B}^\omega_p$, where $1<p<\infty$. These three parts were later  extended to the setting of Fock spaces in \cite{MR3010276} and $W^*(\mathbb{D})$-weighted Bergman spaces in \cite{MR4413302}. Furthermore, the function $\tau_\omega|f'|$ is used in \cite{MR4413302} to characterize the asymptotic behavior of $s_n(H_{\bar{f}})$ in the form \eqref{firstequi}. Building on the aforementioned results, we now investigate the relationship between the characterization functions $\tau_\omega|f'|$ and  $G_{\delta,\omega}(f)$, establishing connections between the ${\rm IDA}$ space (resp. ${\rm BDA}$ space, ${\rm VDA}$ space, the asymptotic behavior of $(G_{\delta,\omega}(f))^*$ and the $\mathcal{B}_p^\omega$ space (resp. $\mathcal{B}^\omega$ space, $\mathcal{B}_0^\omega$ space, the asymptotic behavior of $(\tau_\omega|\bar{f}'|)^*$. Specifically, we will show that
\begin{enumerate}
  \item If $f$ is an anti-analytic symbol, then the characterizations formulated in terms of $G_{\delta,\omega}(f)$ are equivalent to those in terms of $\tau_\omega|\bar{f}'|$. Therefore, Theorems \ref{boundedness},  \ref{compactness} and \ref{SpandSpinfty} are compatible to the characterizations established in \cite{MR4413302};
  \item If $f\in\mathcal{S}$, then the characterizations formulated in terms of $G_{\delta,\omega}(f)$ may no longer be equivalent to those in terms of $\tau_\omega|\partial_z\bar{f}|$, where $\partial_z$ is understood in the weak sense. This suggests that in the general setting $f\in\mathcal{S}$, $G_{\delta,\omega}(f)$ can be regarded as a suitable substitute of $\tau_\omega|\bar{f}'|$.
\end{enumerate}
Inspired by \cite[Lemma 3.3]{MR4668087}, we show in the following lemma that for anti-analytic symbol $\bar{\phi}$, $G_{\delta,\omega}(\bar{\phi})(z)$ has a simpler expression.

	\begin{lemma}\label{attain}
Let $\omega \in \mathcal{W}^\ast(\Omega)$, $\delta\in (0,B_\omega\delta_\omega)$ and $\phi\in {\rm Hol}(\Omega)$, then
$$G_{\delta,\omega}(\bar{\phi})(z)=MO_{\delta,\omega} (\phi)(z),\quad z\in\Omega.$$
\end{lemma}
\begin{proof}	
To begin with, note that for all $h\in \text{Hol}(D(z,\delta \tau_\omega(z)))$, it follows from the mean value property of $(\phi-\phi(z))h$ that
\begin{align*}
			&\int_{D(z,\delta \tau_\omega(z))}({\phi}-{\phi}(z)){h}\,dA=({\phi}(z)-{\phi}(z)){h}(z)=0,\\
			&\int_{D(z,\delta \tau_\omega(z))}(\bar{\phi}-\bar{\phi}(z))\bar{h}\,dA=(\bar{\phi}(z)-\bar{\phi}(z))\bar{h}(z)=0.
		\end{align*}
Applying these two equalities and then expanding $|\bar{\phi}-\bar{\phi}(z)|^2$, we deduce that
		\begin{align*}
			MO_{\delta,\omega} (\bar{\phi})(z)^2&=\fint_{D(z,\delta \tau_\omega(z))}|\bar{\phi}-\bar{\phi}(z)|^2\,dA\\
			&= \inf_{h \in \text{Hol}(D(z,\delta \tau_\omega(z)))} \fint_{D(z,\delta \tau_\omega(z))}|\bar{\phi}-\bar{\phi}(z)-h|^2\,dA\\
			&=\inf_{h \in \text{Hol}(D(z,\delta \tau_\omega(z)))} \fint_{D(z,\delta \tau_\omega(z))}|\bar{\phi}-h|^2\,dA\\
&=G_{\delta,\omega}(\bar{\phi})(z)^2.
		\end{align*}
This, together with a trivial fact $MO_{\delta,\omega} (\bar{\phi})(z)=MO_{\delta,\omega} (\phi)(z)$ for all $z\in\Omega$, ends the proof of Lemma \ref{attain}.
\end{proof}
\begin{proposition}\label{compareProp}
Let $\omega \in \mathcal{W}^\ast(\Omega)$, $\delta\in(0,\delta_\omega)$ and $\phi\in {\rm Hol}(\Omega)$. Then
\begin{enumerate}
  \item $\bar{\phi}\in BDA_\omega(\Omega)$ if and only if $\phi\in \mathcal{B}^\omega$. Moreover, we have
  $$\displaystyle \|\bar{\phi}\|_{BDA_\omega(\Omega)}\asymp \displaystyle \|\phi\|_{\mathcal{B}^\omega};$$
  \item $\ \displaystyle\bar{\phi}\in VDA_\omega(\Omega)$ if and only if $\phi\in \mathcal{B}_0^\omega$.
\end{enumerate}
\end{proposition}
\begin{proof}
		(1) $(\Longrightarrow)$ First, we assume that $\bar{\phi}\in BDA_\omega(\Omega)$. Given $z\in \Omega$, let $\phi_z(w):=\phi(w)-\phi(z)$, then by applying the Cauchy formula for the derivative of $\phi_z$, we have
		\begin{align*}
			\phi'(z)&=\phi_z'(z)\\&=\frac{1}{2\pi \mathrm{i}}\int_{|\xi-z|=t}\frac{\phi_z(\xi)}{(\xi-z)^2}\,d\xi\\
			&=\frac{1}{2\pi}\int_{0}^{2\pi}
			\frac{\phi(z+te^{\mathrm{i}\theta})-\phi(z)}{te^{\mathrm{i}\theta}}\,d\theta,
		\end{align*}
		where in the last equality we used the change of variable $\xi=z+te^{\mathrm{i}\theta}$. Next, by integrating both sides with respect to the measure $t\,dt$ from $\delta\tau_\omega(z)/2$ to $\delta\tau_\omega(z)$, we see that
		\begin{align*} \frac{3}{8}\delta^2\tau_\omega(z)^2\phi'(z)=\frac{1}{2\pi}\int_{D(z,\delta\tau_\omega(z))\backslash D(z,\delta\tau_\omega(z)/2)}
			\frac{\phi(\xi)-\phi(z)}{\xi-z}\,dA(\xi).
		\end{align*}
It follows that
		\begin{align*}
			\tau_\omega(z)^2|\phi'(z)|&\lesssim \int_{D(z,\delta\tau_\omega(z))\backslash D(z,\delta\tau_\omega(z)/2)}
			\left|\frac{\phi(\xi)-\phi(z)}{\xi-z}\right|\,dA(\xi)\\
&\lesssim  \tau_\omega(z)^{-1}\int_{D(z,\delta\tau_\omega(z))}
			\left|\phi(\xi)-\phi(z)\right|\,dA(\xi).
		\end{align*}
By this inequality and Cauchy-Schwarz's inequality, we conclude that
		\begin{align}\label{cb112}\tau_\omega(z)|\phi'(z)|\lesssim MO_{\delta,\omega} (\bar{\phi})(z)=G_{\delta,\omega}(\bar{\phi})(z).
\end{align}
Combining inequalities \eqref{cb111} with \eqref{cb112}, we complete the proof of the first statement.
		
		$(\Longleftarrow)$ Next, we assume that $\phi\in \mathcal{B}^\omega$. To show the required result, we first note that for all $z\in \Omega$ and $w\in D(z,\delta\tau_\omega(z))$,  $$\tau_\omega(z)\asymp \tau_\omega(w)\asymp \tau_\omega(z+t(w-z)),\quad {\rm for}\ {\rm all}\ 0\leq t\leq 1.$$ It follows that
		\begin{align}\label{bdca}
			|\phi(w)-\phi(z)|&=\left|\int_0^1 \phi'(z+t(w-z))(w-z)\,dt\right|\nonumber\\
			&\leq \int_0^1 \left|\phi'(z+t(w-z))(w-z)\right|\,dt\nonumber\\
			&\leq \delta \tau_\omega(z) \tau_\omega(z+t(w-z))^{-1}\sup\limits_{\xi\in D(z,\delta \tau_\omega(z))}\tau_\omega(\xi)|\phi'(\xi)|\nonumber\\
			&\lesssim \sup\limits_{\xi\in D(z,\delta \tau_\omega(z))}\tau_\omega(\xi)|\phi'(\xi)|.
		\end{align}
Since $\phi\in \text{Hol}(\Omega)$, by the mean value property of $\phi$, we have $\hat{\phi}_\delta (z)=\phi(z)$. Combining this equality with inequality \eqref{bdca} and Lemma \ref{attain}, we conclude that
		\begin{align}\label{cb111}
G_{\delta,\omega}(\bar{\phi})(z)&=MO_{\delta,\omega} (\phi)(z)\nonumber\\
&=\left(\fint_{D(z,\delta\tau_\omega(z))}|\phi(w)-\phi(z)|^2\, dA(w)\right)^{1/2}\nonumber\\
&\lesssim \sup\limits_{\xi\in D(z,\delta \tau_\omega(z))}\tau_\omega(\xi)|\phi'(\xi)|.
\end{align}
This implies the required result directly.
		
		(2) $(\Longrightarrow)$ First, we assume that $\displaystyle\bar{\phi}\in VDA_\omega(\Omega)$. Under this assumption, the required result follows directly from inequality \eqref{cb112}.

$(\Longleftarrow)$ Next, we assume that $\phi\in \mathcal{B}_0^\omega$.  This assumption, together with assumption ${\rm\bf (H)_3}$, yields
$$\lim_{z\rightarrow \partial_\infty\Omega}\sup\limits_{\xi\in D(z,\delta \tau_\omega(z))}\tau_\omega(\xi)|\phi'(\xi)|=0.$$
Combining this fact with inequality \eqref{cb111}, we deduce that
$$\lim_{z\rightarrow \partial_\infty\Omega}G_{\delta,\omega}(\bar{\phi})(z)=0.$$
We complete the proof of Proposition \ref{compareProp}.
	\end{proof}
\begin{remark}\label{keyremark000}
If $f\notin {\rm Hol}(\Omega)$, then the norm equivalences
$$\displaystyle \|f\|_{BDA_\omega(\Omega)}\asymp \displaystyle \|\tau_\omega\partial_z\bar{f}\|_{L^\infty(\Omega)}$$
and
$$\displaystyle \|f\|_{IDA_\omega^p(\Omega)}\asymp \displaystyle \|\tau_\omega\partial_z\bar{f}\|_{L^p(\Omega,d\lambda_\omega)},\quad 1<p<\infty,$$
may not be true, where $\partial_z$ is understood in the weak sense. This can be seen by considering the following function in the context of standard Fock space:
\begin{align}\label{counter000}
f(z):=\begin{cases}
\frac{1}{z},& {\rm if}\ |z|\geq 1,\\
0,& {\rm if}\ |z|< 1.
\end{cases}
\end{align}
Indeed, the above function satisfies $0<G_{1,\omega}(f)(z)\lesssim 1$ when $0<|z|< 2$, but $$|\partial_z\bar{f}(z)|=0, \quad {\rm for}\ {\rm all}\ z\in\mathbb{C}\backslash \partial\mathbb{D}.$$
See {\rm \cite{MR4552558,MR4744482}} for a discussion about the properties of this function.
\end{remark}

\begin{proposition}\label{compareProp2}
Let $\omega \in \mathcal{W}^\ast(\Omega)$, $\delta\in (0,\delta_\omega)$ and $\phi\in {\rm Hol}(\Omega)$.
\begin{enumerate}
  \item Let $h: [0,+\infty)\to  [0,+\infty)$ be an increasing function such that $h(0)=0$ and $h(t^p)$ is convex for some $0<p<\infty$, then there exists $B>0$, which depends only on $\delta$, $\omega$ and $p$, such that
	$$
	\displaystyle \int _\Omega h\left ( \frac{1}{B}G_{\delta,\omega}(\bar{\phi})(z)\right ) d\lambda _\omega (z) \leq  \displaystyle \int _\Omega h\left ( \tau_\omega(z)|\phi'(z)|\right ) d\lambda _\omega (z) \leq \displaystyle \int _\Omega h\left ( BG_{\delta,\omega}(\bar{\phi})(z)\right ) d\lambda _\omega (z).
	$$
Consequently, for any  $1\leq p<\infty$, $\bar{\phi}\in IDA_\omega^p(\Omega)$ if and only if $\phi\in \mathcal{B}_p^\omega$. Moreover, we have
$$\|\bar{\phi}\|_{IDA_\omega^p(\Omega)}\asymp \|\phi\|_{\mathcal{B}_p^\omega}.$$
  \item If $\rho $ is an increasing function  such that $\rho (x)/x^ \gamma$ is decreasing for some $\gamma>0$, then
      $$(G_{\delta,\omega}(\bar{\phi}))^*(n)  \lesssim 1/ \rho (n),\ {\rm for}\ {\rm all}\ n\in\mathbb{N}  \Longleftrightarrow  (\tau_\omega|\phi'|)^*(n)  \lesssim 1/ \rho (n),\ {\rm for}\ {\rm all}\ n\in\mathbb{N}.$$
\end{enumerate}
\end{proposition}
\begin{proof}
(1) The second inequality is a direct consequence of inequality \eqref{cb112} in combination with the condition that $h$ is an increasing function. Therefore, it remains to show the first inequality. To this end, we first apply the subharmonicity of $|\phi '|^{1/p}$ to get that
		$$|\phi '(\xi)|\leq  \left(\fint_{D(\xi,\delta\tau_\omega(\xi))}|\phi '|^{1/p}\,dA\right)^p,\quad {\rm for}\ {\rm all}\ \xi\in\Omega.$$
Recall from Lemma \ref{lattice} and condition ${\rm\bf (H)_4}$ that $D(\xi, \delta\tau_\omega(\xi))\subset D(z,B_\omega \delta\tau_\omega(z))$ and $\tau_\omega(\xi)\asymp \tau_\omega(z)$ for all $\xi \in D(z,\delta \tau_\omega(z))$. Combining these facts with Lemma \ref{attain} and inequality \eqref{cb111}, we see that
		\begin{align}\label{qwerytu}
G_{\delta,\omega}(\bar{\phi})(z)&=MO_{\delta,\omega}(\phi)(z)\nonumber\\&\lesssim \sup\limits_{\xi\in D(z,\delta \tau_\omega(z))}\tau_\omega(\xi)|\phi'(\xi)|\nonumber\\
&\lesssim \tau_\omega(z)^{1-2p}\left(\int_{D(z,B_\omega \delta\tau_\omega(z))}|\phi '|^{1/p}\,dA\right)^p.
\end{align}
		By condition ${\rm\bf (H)_4}$, we see that
		$|k_{z}(u)|\asymp \|K_u\|_{A_\omega^2}=\tau_\omega(u)^{-1}\omega(u)^{-1/2}$ and  $\tau_\omega(u)\asymp\tau_\omega(z)$ for all $u \in D(z,B_\omega\delta \tau_\omega(z))$. Hence,
		\begin{align}\label{cbb1234}
			\tau_\omega(z)^{1-2p}\left(\int_{D(z,B_\omega \delta\tau_\omega(z))}|\phi '(u)|^{1/p}\,dA(u)\right)^p&\asymp \left(\int_{D(z,B_\omega \delta\tau_\omega(z))}(\tau_\omega(u)|\phi '(u)|)^{1/p}\cdot |k_z(u)|^2\,dA_\omega(u)\right)^p\nonumber\\
			&\lesssim \left(\int_{\Omega}(\tau_\omega(u)|\phi '(u)|)^{1/p}\cdot |k_z(u)|^2\,dA_\omega(u)\right)^p.
		\end{align}
Since $h$ is an increasing function, we combine inequalities \eqref{qwerytu} with  \eqref{cbb1234} to conclude that there exists a constant $B>0$ such that
		$$h(\frac{1}{B}G_{\delta,\omega}(\bar{\phi})(z))\leq h\left[\left(\int_{\Omega}(\tau_\omega(u)|\phi '(u)|)^{1/p}\cdot |k_z(u)|^2\,dA_\omega(u)\right)^p\right].$$
		Since $|k_z|^2\,dA_\omega$ is a probability measure for each $z\in \Omega$ and $h(t^p)$ is a convex function for some $0<p<\infty$, we apply Jensen's inequality to deduce that
		$$h\left[\left(\int_{\Omega}(\tau_\omega(u)|\phi '(u)|)^{1/p}\cdot |k_z(u)|^2\,dA_\omega(u)\right)^p\right] \leq \int_{\Omega}h(\tau_\omega(u)|\phi '(u)|)\cdot |k_z(u)|^2\,dA_\omega(u).$$
To continue, we apply the equality $\|K_z\|_{A_\omega^2}= \tau_\omega(z)^{-1}\omega(z)^{-1/2}$ to get that
\begin{align*}
			\int_{\Omega}|k_z(u)|^2\,d\lambda_\omega(z)&=\int_{\Omega}|K(z,u)|^2\|K_{z}\|_{A_\omega^2}^{-2}\tau_\omega(z)^{-2}\,dA(z)\\
&=\int_{\Omega}|K(z,u)|^2\,dA_\omega(z)\\
			&=\tau_\omega(u)^{-2}\omega(u)^{-1}.
		\end{align*}
This, in combination with Fubini's Theorem, yields
\begin{align*}
\int_{\Omega}h(\frac{1}{B}G_{\delta,\omega}(\bar{\phi})(z))\,d\lambda_\omega(z)
&\leq \int_{\Omega}\int_{\Omega}h(\tau_\omega(u)|\phi '(u)|)\cdot |k_z(u)|^2\,dA_\omega(u)\,d\lambda_\omega(z) \\
&=\int_{\Omega}h(\tau_\omega(u)|\phi '(u)|)\,d\lambda_\omega(u).
\end{align*}
This finishes the proof of the first statement in (1), while the second one follows by choosing $h(t)=t^p$ directly.

To show the second statement, we first apply \cite[Proposition 1.1.4 and Proposition 1.4.5 (12)]{MR3243734} to see that for every increasing  function $h: [0,+\infty)\to  [0,+\infty)$ satisfying $h(0)=0$,
$$\int_{\Omega} h(G_{\delta,\omega}(\bar{\phi})(z))\,d\lambda_\omega(z)=\int_0^\infty h((G_{\delta,\omega}(\bar{\phi}))^*(t))dt,$$
$$\int_{\Omega} h(\tau_\omega(z)|\phi'(z)|)\,d\lambda_\omega(z)=\int_0^\infty h((\tau_\omega|\phi'|)^*(t))dt.$$
This, in combination with the first statement in (1), yields that for every increasing  function $h: [0,+\infty)\to  [0,+\infty)$ satisfying $h(0)=0$ and $h(t^p)$ is convex for $p>\gamma$, we have
\begin{align*}
\int_0^\infty h(\frac{1}{B}(G_{\delta,\omega}(\bar{\phi}))^*(t))dt\lesssim \displaystyle \int_0^\infty h((\tau_\omega|\phi'|)^*(t))dt \lesssim \int_0^\infty h(B(G_{\delta,\omega}(\bar{\phi}))^*(t))dt.
\end{align*}
This, together with Remark \ref{equiremark} and Lemma \ref{finelemma}, yields
$$(G_{\delta,\omega}(\bar{\phi}))^*(n)  \lesssim 1/ \rho (n),\ {\rm for}\ {\rm all}\ n\in\mathbb{N}  \Longleftrightarrow  (\tau_\omega|\phi'|)^*(n)  \lesssim 1/ \rho (n),\ {\rm for}\ {\rm all}\ n\in\mathbb{N}.$$
This completes the proof of Proposition \ref{compareProp2}.
\end{proof}

%
%

\section{Berger-Coburn phenomenon for weak Schatten-$p$ class}\label{BCph}
This section is devoted to providing a proof for Theorem \ref{Coburn}.

\begin{definition}
Let $r>0$ and $\mu$ be a positive Borel measure on $\mathbb{C}$, we define the average function $\hat{\mu}_r$ on $\mathbb{C}$ by
\begin{align*}
\hat{\mu}_r(z):=\frac{\mu(D(z,r))}{A(D(z,r))}.
\end{align*}
\end{definition}
\begin{lemma}\label{ToeLemma}
Let $\mu$ be a positive Borel measure on $\mathbb{C}$. Suppose that $\varphi\in C^2(\mathbb{C})$ is real-valued with $\mathrm{i} \partial \bar{\partial} \varphi \simeq \omega_0$, where $\omega_0=i\partial\bar{\partial}|z|^2$ is the Euclidean-K\"{a}hler form on $\mathbb{C}$. Then for any $1< p<\infty$ and $r>0$, there exists a constant $C>0$ such that
\begin{align*}
\|\langle T_\mu k_z,k_z\rangle_{F^2(\varphi)}\|_{L^{p,\infty}(\mathbb{C})}\leq C \|\hat{\mu}_r\|_{L^{p,\infty}(\mathbb{C})}.
\end{align*}
\end{lemma}
\begin{proof}
We apply formula \eqref{Toeformula}, \cite[Lemma 4.1]{MR4402674} and \cite[Formula (2.4)]{MR4402674} to deduce that for any $r>0$, there are positive constants $\theta>0$ and $C>0$ such that
\begin{align*}
\langle T_\mu k_z,k_z\rangle_{F^2(\varphi)}
&=\int_{\mathbb{C}}|k_z(\xi)|^2e^{-\varphi(\xi)}d\mu(\xi)\\
&\leq C \int_{\mathbb{C}}|k_z(\xi)|^2e^{-\varphi(\xi)}\hat{\mu}_r(\xi)dA(\xi)\\
&\leq C \int_{\mathbb{C}}e^{-\theta|z-\xi|}\hat{\mu}_r(\xi)dA(\xi)\\
&=C ({\rm Exp}_{\theta}\ast \hat{\mu}_r)(z),
\end{align*}
where ${\rm Exp}_{\theta}(z):=e^{-\theta |z|}\in L^1(\mathbb{C})$ and $\ast $ denotes the convolution on $\mathbb{C}$ (see \eqref{defconv}).
This, in combination with weak-type convolution Young's inequality \cite[Theorem 1.2.13]{MR3243734}, yields
\begin{align*}
\|\langle T_\mu k_z,k_z\rangle_{F^2(\varphi)}\|_{L^{p,\infty}(\mathbb{C})}\lesssim \|{\rm Exp}_{\theta}\|_{L^1(\mathbb{C})} \|\hat{\mu}_r\|_{L^{p,\infty}(\mathbb{C})}\lesssim \|\hat{\mu}_r\|_{L^{p,\infty}(\mathbb{C})}.
\end{align*}
This finishes the proof of Lemma \ref{ToeLemma}.
\end{proof}

The following result, together with Theorem \ref{main0} where  $\rho(t)$ is chosen to be $(1+t)^{1/p}$, establishes a new necessary and sufficient condition for $H_f\in S^{p,\infty}$ ($1<p<\infty$) in the context of weighted Fock space.
\begin{proposition} \label{Berezin}
Let $f\in \mathcal S$ and suppose that $\varphi\in C^2(\mathbb{C})$ is real valued with $\mathrm{i} \partial \bar{\partial} \varphi \simeq \omega_0$, where $\omega_0=i\partial\bar{\partial}|z|^2$ is the Euclidean-K\"{a}hler form on $\mathbb{C}$. Then for any $0<p<\infty$, $f\in \mathrm{IDA}^{p,\infty}(\mathbb{C})$ if and only if $\|H_f(k_z)\|_{L^2(\varphi)}\in L^{p,\infty}(\mathbb{C})$.
\end{proposition}
\begin{proof}
$(\Longrightarrow)$ Suppose that $f\in IDA^{p,\infty}(\mathbb{C})$. By Lemma \ref{decomposelemma} (see also \cite[Lemma 3.6]{MR4668087} and \cite[Section 2.2]{MR4668087} the lattices in $\mathbb{C}$, which is slightly different from the one given in Lemma \ref{decomposelemma}), $f$ admits a decomposition $f=f_1+f_2$ with
\begin{equation}\label{decomp-add}
	|\bar \partial f_1(z)|+ M_{\delta/2}(|\bar \partial f_1|)(z)  + M_{ \delta/2}(f_2)(z) \le C G_{\delta}(f)(z), \ z\in \mathbb{C},
\end{equation}
where $M_{\delta}(f)$ is defined in \eqref{defM}, with $\tau_\omega(z)$ being replaced by $1$.

{\bf Case 1: $0<p\leq 2$.}   Recall in \cite[Formula (4.4)]{MR4402674} that for any $g\in \Gamma$,
\begin{align}\label{44}
\|H_{f_1}g\|_{L^2(\varphi)}\lesssim \|g\bar{\partial}f_1\|_{L^2(\varphi)}.
\end{align}
We apply \cite[Lemma 4.1]{MR4402674} with $d\mu=|\bar{\partial}f_1|^2dA$ to deduce that for any $\delta>0$, there exists a constant $C>0$ such that
\begin{align}\label{evbs}
\left\|\|H_{f_1}(k_z)\|_{L^2(\varphi)}\right\|_{L^{p,\infty}(\mathbb{C})}&\leq C\left\|\left(\int_{\mathbb{C}}\left|k_z(\xi)e^{-\frac{\varphi(\xi)}{2}}\right|^2|\bar{\partial}f_1(\xi)|^2dA(\xi)\right)^{1/2}\right\|_{L^{p,\infty}(\mathbb{C})}\nonumber\\
&\leq C\left\|\left(\int_{\mathbb{C}}\left|k_z(\xi)e^{-\frac{\varphi(\xi)}{2}}\right|^{\frac p2}M_{\delta/2}(\bar{\partial}f_1)(\xi)^{\frac p2}dA(\xi)\right)^{2/p}\right\|_{L^{p,\infty}(\mathbb{C})}.
\end{align}
To continue, we recall from \cite[Formula (2.4)]{MR4402674} that there is a positive constant $\theta>0$ such that
\begin{align}\label{kernelestimate}
|k_z(\xi)e^{-\frac{\varphi(\xi)}{2}}|\lesssim e^{-\theta|z-\xi|}.
\end{align}
Substituting the above kernel estimate into \eqref{evbs} and then applying weak-type convolution Young's inequality \cite[Theorem 1.2.13]{MR3243734}, we deduce that the right-hand side of \eqref{evbs} is dominated by
\begin{align}
\left\|\left(\int_{\mathbb{C}}e^{-\frac{p\theta|z-\xi|}{2}}M_{\delta/2}(\bar{\partial}f_1)(\xi)^{\frac{p}{2}}dA(\xi)\right)^{2/p}\right\|_{L^{p,\infty}(\mathbb{C})}
&=\left\|{\rm Exp}_{\frac{p\theta}{2}}\ast M_{\delta/2}(\bar{\partial}f_1)^{\frac p2} \right\|_{L^{2,\infty}(\mathbb{C})}^{2/p}\nonumber\\
&\lesssim \left\|{\rm Exp}_{\frac{p\theta}{2}}\right\|_{L^1(\mathbb{C})}^{2/p}\left\| M_{\delta/2}(\bar{\partial}f_1)^{\frac{p}{2}}\right\|_{{L^{2,\infty}(\mathbb{C})}}^{2/p}\nonumber\\
&\lesssim \left\| M_{\delta/2}(\bar{\partial}f_1)\right\|_{{L^{p,\infty}(\mathbb{C})}},
\end{align}
where ${\rm Exp}_{\frac{p\theta}{2}}(z):=e^{-\frac{p\theta |z|}{2}}\in L^1(\mathbb{C})$ and $\ast $ denotes the convolution on $\mathbb{C}$ (see \eqref{defconv}). Hence,
\begin{align}\label{cbf1}
\left\|\|H_{f_1}(k_z)\|_{L^2(\varphi)}\right\|_{L^{p,\infty}(\mathbb{C})}\lesssim \| M_{\delta/2}(\bar{\partial}f_1)\|_{{L^{p,\infty}(\mathbb{C})}}.
\end{align}
Next, note that
$$\|H_{f_2}g\|_{L^2(\varphi)}\lesssim \|gf_2\|_{L^2(\varphi)}.$$
Similarly, we apply \cite[Lemma 4.1]{MR4402674} with $d\mu=|f_2|^2dA$ to deduce that
\begin{align}\label{evbs00}
\left\|\|H_{f_2}(k_z)\|_{L^2(\varphi)}\right\|_{L^{p,\infty}(\mathbb{C})}
&\lesssim\left\|\left(\int_{\mathbb{C}}\left|k_z(\xi)e^{-\frac{\varphi(\xi)}{2}}\right|^2|f_2(\xi)|^2dA(\xi)\right)^{1/2}\right\|_{L^{p,\infty}(\mathbb{C})}\nonumber\\
&\lesssim\left\|\left(\int_{\mathbb{C}}\left|k_z(\xi)e^{-\frac{\varphi(\xi)}{2}}\right|^{\frac p2}M_{\delta/2}(f_2)(\xi)^{\frac p2}dA(\xi)\right)^{2/p}\right\|_{L^{p,\infty}(\mathbb{C})}.
\end{align}
Substituting the estimate \eqref{kernelestimate} into \eqref{evbs00} and then applying weak-type convolution Young's inequality \cite[Theorem 1.2.13]{MR3243734}, we see that the right-hand side of \eqref{evbs00} is dominated by
\begin{align}\label{vbjiong}
\left\|\left(\int_{\mathbb{C}}e^{-\frac{p\theta|z-\xi|}{2}}M_{\delta/2}(f_2)(\xi)^{\frac p2}dA(\xi)\right)^{2/p}\right\|_{L^{p,\infty}(\mathbb{C})}
&=\left\|{\rm Exp}_{\frac{p\theta}{2}}\ast M_{\delta/2}(f_2)^{\frac p2} \right\|_{L^{2,\infty}(\mathbb{C})}^{2/p}\nonumber\\
&\lesssim \left\|{\rm Exp}_{\frac{p\theta}{2}}\right\|_{L^1(\mathbb{C})}^{2/p}\| M_{\delta/2}(f_2)^{\frac{p}{2}}\|_{{L^{2,\infty}(\mathbb{C})}}^{2/p}\nonumber\\
&\lesssim \left\| M_{\delta/2}(f_2)\right\|_{{L^{p,\infty}(\mathbb{C})}}.
\end{align}
This, in combination with inequality \eqref{decomp-add}, yields
\begin{align}\label{ssmmt}
\left\|\|H_{f_2}(k_z)\|_{L^2(\varphi)}\right\|_{L^{p,\infty}(\mathbb{C})}\lesssim \| M_{\delta/2}(f_2)\|_{{L^{p,\infty}(\mathbb{C})}}.
\end{align}

By inequality \eqref{decomp-add}, the right-hand sides of inequalities \eqref{cbf1} and \eqref{ssmmt} are bounded by $C\|G_{\delta}(f)\|_{{L^{p,\infty}(\mathbb{C})}}$. Therefore, $\|H_f(k_z)\|_{L^2(\varphi)}\in L^{p,\infty}(\mathbb{C})$.

{\bf Case 2: $2<p<\infty$.} By \eqref{44},
\begin{align}\label{later1}
	\|H_{f_1}(k_z)\|_{L^2(\varphi)}^p
	\lesssim \left \langle  |\bar \partial f_1|^2 k_z, k_z  \right\rangle ^{\frac p 2}
	=C \left \langle T_{\mu_{f_1}} k_z, k_z  \right\rangle_{F^2(\varphi)} ^{\frac p 2},
\end{align}
where $d\mu_{f_1}:=|\bar \partial f_1|^2 dA$.
Moreover, we apply Lemma \ref{ToeLemma} to deduce that
\begin{align*}
\left\|\left \langle  T_{\mu_{f_1}} k_z, k_z  \right\rangle_{F^2(\varphi)}\right\|_{L^{p/2,\infty}(\mathbb{C})}^{\frac p 2}
\lesssim  \left\|M_{\delta/2}( |\bar{\partial} f_1|)^2\right\|_{L^{p/2,\infty}(\mathbb{C})}^{\frac p2}=C\left\|M_{\delta/2}( |\bar{\partial} f_1|)\right\|_{L^{p,\infty}(\mathbb{C})}^{p}.
\end{align*}
This, in combination with inequality \eqref{later1}, yields
\begin{align}\label{hff12}
\left\|\|H_{f_1}(k_z)\|_{L^2(\varphi)}\right\|_{L^{p,\infty}(\mathbb{C})}^p
&=\left\|\|H_{f_1}(k_z)\|_{L^2(\varphi)}^p\right\|_{L^{1,\infty}(\mathbb{C})}\nonumber\\
&\lesssim \left\|\left \langle  T_{ \mu_{f_1}} k_z, k_z  \right\rangle_{F^2(\varphi)}^{\frac p 2}\right\|_{L^{1,\infty}(\mathbb{C})}\nonumber\\
&=C \left\|\left \langle  T_{ \mu_{f_1}} k_z, k_z  \right\rangle_{F^2(\varphi)}\right\|_{L^{p/2,\infty}(\mathbb{C})}^{\frac p 2}\nonumber\\
&\lesssim \left\|M_{\delta/2}( |\bar{\partial} f_1|)\right\|_{L^{p,\infty}(\mathbb{C})}^{p}.
\end{align}

For the term involving $f_2$, we note that
\begin{align}\label{later2}
	\|H_{f_2}(k_z)\|_{L^2(\varphi)}^p  \le \left \langle  |  f_2|^2 k_z, k_z  \right\rangle_{F^2(\varphi)} ^{\frac p 2}=C \left \langle T_{\mu_{f_2}} k_z, k_z  \right\rangle_{F^2(\varphi)} ^{\frac p 2},
\end{align}
where $d\mu_{f_2}:=|f_2|^2dA$. By repeating the argument in the proof of inequality \eqref{hff12}, we see that
\begin{align}\label{hff13}
\left\|\|H_{f_2}(k_z)\|_{L^2(\varphi)}\right\|_{L^{p,\infty}(\mathbb{C})}
&\lesssim \left\|M_{\delta/2}( |f_2|)\right\|_{L^{p,\infty}(\mathbb{C})}.
\end{align}

By inequality \eqref{decomp-add}, the right-hand sides of inequalities \eqref{hff12} and \eqref{hff13} are bounded by $C\|G_{\delta}(f)\|_{{L^{p,\infty}(\mathbb{C})}}$. Therefore, $\|H_f(k_z)\|_{L^2(\varphi)}\in L^{p,\infty}(\mathbb{C})$.

$(\Longleftarrow)$ Suppose that $\|H_f(k_z)\|_{L^2(\varphi)}\in L^{p,\infty}(\mathbb{C})$. Then it follows from \cite[Formula (4.7)]{MR4402674} that $G_\delta(f)(z)\lesssim \|H_f(k_z)\|_{L^2(\varphi)}$ for all $z\in\mathbb{C}$. This implies directly that
$$\|f\|_{IDA^{p,\infty}(\mathbb{C})}\lesssim \left\|\|H_f(k_z)\|_{L^2(\varphi)}\right\|_{ L^{p,\infty}(\mathbb{C})}.$$
This ends the proof of Proposition \ref{Berezin}.
\end{proof}

Let $\mathcal{B}$ be the Ahlfors-Beurling operator, which is defined as:
$$\mathcal{B}(f)(z):={\rm p.v.}-\frac{1}{\pi}\int_{\mathbb{C}}\frac{f(\xi)}{(\xi-z)^2}dA(\xi),$$
where ${\rm p.v.}$ denotes the Cauchy principle value (See \cite{MR2241787,MR2472875} for more details). Ahlfors-Beurling operator is a Calder\'{o}n-Zygmund operator, which means that it is bounded on $L^p(\mathbb{C})$ for all $1<p<\infty$. Furthermore, by real interpolation (see \cite[Theorems 3.11.8 and 5.2.1]{Bergh}), it is also bounded on $L^{p,\infty}(\mathbb{C})$.

The following lemma is a weak-type version of \cite[Lemma 7.1]{MR4402674}.
\begin{lemma}\label{barthm}
Suppose that $1<p<\infty$. Then there is a constant $C>0$ depending only on $p$ such that for all $f\in C^2(\mathbb{C})\cap L^\infty(\mathbb{C})$,
$$\left\|\frac{\partial f}{\partial z}\right\|_{L^{p,\infty}(\mathbb{C})}\leq C\left\|\frac{\partial f}{\partial \bar{z}}\right\|_{L^{p,\infty}(\mathbb{C})}.$$
\end{lemma}
\begin{proof}
To begin with, we choose $\psi\in C^\infty(\mathbb R)$ be a decreasing function such that $\psi(x)=1$ for $x\le 0$, $\psi(x)=0$ for $x\ge 1$, and $0\le -\psi'(x)\le 2$ for $x\in  \mathbb R$.
		Define $\psi_R(x):= \psi(x-R)$ and $ f_R(z):=  f(z) \psi_R(|z|)$, where $R>0$ to be chosen later. Then $ f_R(z) \in C^2_c(\mathbb C)$. It follows from \cite[Theorem 2.1.1]{MR1800297} that
		$$
		f_R (z)= \frac 1{2\pi \mathrm i} \int_{\mathbb C} \frac {\frac {\partial f_R}{\partial \bar z}}{\xi-z} d\xi\wedge d\bar \xi.
		$$
		Note that $ \frac {\partial f_R}{\partial \bar z} = \psi_R \frac {\partial  f}{\partial \bar z}+ f \frac {\partial  \psi_R}{\partial \bar z} $.   By \cite[Lemma 2]{MR2241787},
		\begin{equation}\label{partial-estimate-z}
			\frac {\partial f_R}{\partial z}(z)
			= \mathcal{B}\left( \frac {\partial f_R}{\partial \bar z}\right)(z)
			= \mathcal{B}\left(\psi_R \frac {\partial  f}{\partial \bar z}\right)(z)+ \mathcal{B}\left(f \frac {\partial  \psi_R}{\partial \bar z}\right)(z).
		\end{equation}
		Since the support of $\frac {\partial  \psi_R}{\partial \bar z}$ is contained in $\{\xi \in \mathbb{C}:R\leq |\xi|\leq R+1\}$, we deduce that for all $r>0$ and $|z|<r$, if $R$ is chosen to be sufficiently large, then
		$$
		\left|\mathcal{B}\left(f \frac {\partial  \psi_R}{\partial \bar z}\right)(z)\right| \le \frac{\|f\|_{L^\infty(\mathbb{C})} } {\pi(R-r)^2}  \int_{R\le |\xi|\le R+1}   dA(\xi)\le  \frac{3R \|f\|_{L^\infty(\mathbb{C})} } { (R-r)^2},
		$$
		and hence,
		\begin{equation}\label{partial-deri-a}
				\left\|\chi_{D(0,r)} \mathcal{B}\left(f \frac {\partial  \psi_R}{\partial \bar z}\right)\right\|_{L^{p,\infty}(\mathbb{C})}^p\leq \left\|\chi_{D(0,r)} \mathcal{B}\left(f \frac {\partial  \psi_R}{\partial \bar z}\right)\right\|_{L^p(\mathbb{C})}^p
			\lesssim  \frac{R^pr^2 \|f\|_{L^\infty(\mathbb{C})}^p } { (R-r)^{2p}}.
		\end{equation}
		 In particular, if we choose $r>1$ such that $r^2\geq r$ and take $R=2r^2$, then
		\begin{equation}\label{partial-deri-aa}
			\left\|\chi_{D(0,r)} \mathcal{B}\left(f \frac {\partial  \psi_{2r^2}}{\partial \bar z}\right)\right\|_{L^{p,\infty}(\mathbb{C})}^p
			\lesssim  r^{2-2p} \|f\|_{L^\infty(\mathbb{C})}^p.
		\end{equation}
		On the other hand, since $\mathcal{B}$ is bounded on $L^{p,\infty}(\mathbb{C})$, we have
		\begin{equation}\label{partial-deri-b}
			\left\| \mathcal{B}\left(\psi_R \frac {\partial  f}{\partial \bar z}\right)\right\|_{L^{p,\infty}(\mathbb{C})}
			\lesssim \left\|  \psi_R \frac {\partial  f}{\partial \bar z} \right\|_{L^{p,\infty}(\mathbb{C})}
			\lesssim \left\|\frac {\partial  f}{\partial \bar z} \right\|_{L^{p,\infty}(\mathbb{C})}.
		\end{equation}
		Combining inequalities \eqref{partial-estimate-z}, \eqref{partial-deri-aa} with \eqref{partial-deri-b}, we conclude that
		$$
		\left\|  \chi_{D(0,r)}\frac {\partial f}{\partial z} \right\|_{L^{p,\infty}(\mathbb{C})}
		= \left\| \chi_{D(0,r)} \frac {\partial f_{2r^2}}{\partial z} \right\|_{L^{p,\infty}(\mathbb{C})}
		\lesssim \left\|\frac {\partial  f}{\partial \bar z} \right\|_{L^{p,\infty}(\mathbb{C})}+r^{2-2p}\|f\|_{L^\infty(\mathbb{C})}.
		$$
		This, in combination with \cite[Exercise 1.1.12]{MR3243734}, yields
		\begin{align}\label{derivative-est}
			\left\|  \frac {\partial f}{\partial z} \right\|_{L^{p,\infty}(\mathbb{C})}
&\leq \liminf_{r\rightarrow \infty}\left\| \chi_{D(0,r)} \frac {\partial f}{\partial z} \right\|_{L^{p,\infty}(\mathbb{C})}\nonumber\\
&\lesssim \liminf_{r\rightarrow \infty}\left(\left\|\frac {\partial  f}{\partial \bar z} \right\|_{L^{p,\infty}(\mathbb{C})}+r^{2-2p}\|f\|_{L^\infty(\mathbb{C})}\right)\nonumber\\
			&=C \left\|\frac {\partial  f}{\partial \bar z} \right\|_{L^{p,\infty}(\mathbb{C})}.
		\end{align}
This ends the proof of Lemma \ref{barthm}.
\end{proof}
Next, we follow the argument in \cite[Theorem 1.2]{MR4402674} (see also \cite{MR4630767}) to establish the Berger-Coburn phenomenon for $S^{p,\infty}$.
\begin{proof}[Proof of Theorem \ref{Coburn}]
Suppose $H_f\in S^{p,\infty}$ with $1<p<\infty$, then it follows from Theorem \ref{main} \eqref{AAAAA} that $f\in \mathrm{IDA}^{p,\infty}(\mathbb{C})$.
Let $\{\psi_n\}$ be a partition of unit given in Lemma \ref{cover}, which is subordinated to the covering $\{D(z_n,\delta/8)\}_n$ of $\mathbb{C}$. Define the $q$-th mean of $|f|$ over $D(z, \delta)$ by
$$
	M_{q, \delta}(f)(z)= \left(\fint_{D(z,\delta)} |f|^q dA \right)^{\frac{1}{q}},\quad z\in\mathbb{C}.
$$
By \cite[Lemma 3.3]{MR4668087}, there is an $h_j\in {\rm Hol}(D(z_j, \delta/4))$ such that $\sup_{z\in D(z_j, \delta/8)}|h_j(z)|\lesssim \|f\|_{ L^\infty(\mathbb{C})}$ and
$$
      M_{2, \delta/4}(f-h_j)(z_j)= G_{\delta/4}(f)(z_j).
$$
By \cite[Lemma 3.5]{MR4668087} (see also Lemma \ref{decomposelemma}), $f$ admits a decomposition $f=f_1+f_2$ with
\begin{equation}\label{decomp-a}
	|\bar \partial f_1(z)|+ M_{2, \delta/4}(|\bar \partial f_1|)(z)  + M_{2, \delta/4}(f_2)(z) \lesssim G_{\delta/2}(f)(z), \quad z\in \mathbb{C},
\end{equation}
where we may choose
$$
	f_1=\sum_{j=1}^\infty h_j\psi_j
$$
as in the proof of Lemma \ref{decomposelemma}. Note that $f_1\in L^\infty(\mathbb{C})$ and $\bar{\partial}\bar{f}_1= F+H$, where
$$
	F= \sum_{j=1}^\infty  \bar {h}_j  \bar {\partial}\psi_j
	\quad{\rm and}\quad
	H= \sum_{j=1}^\infty  \psi_j \bar {\partial}\,\bar  {h}_j.
$$
With the same proof as the one for the estimate of $|\bar \partial  f_1|$ in \eqref{decomp-a}, we have
\begin{equation}\label{est-F}
	|F(z)|\lesssim G_{\delta/2}(f)(z).
\end{equation}
This, in combination with Lemma \ref{barthm}, yields
\begin{equation}\label{f-Sp-norm-b}
	\left \| H \right \|_{L^{p,\infty}(\mathbb{C})} \lesssim\|\bar \partial \bar f_1 \|_{L^{p,\infty}(\mathbb{C})} + \|F \|_{L^{p,\infty}(\mathbb{C})}
	\lesssim \|\bar \partial f_1 \|_{L^{p,\infty}(\mathbb{C})}+ \|F \|_{L^{p,\infty}(\mathbb{C})}\lesssim \|G_{\delta/2}(f)\|_{L^{p,\infty}(\mathbb{C})}.
\end{equation}

To continue, we {\it claim} that
\begin{align}\label{claimin}
M_{2, \delta/16}(H)(z)\lesssim G_{\delta}(f)(z) + M_{p/2, \delta/8}( H)(z),\quad z\in\mathbb{C}.
\end{align}
To show inequality \eqref{claimin}, we first note that for $z\in D(z_j, 3\delta/16)\cap D(z_k, 3\delta/16)$,  the coefficients $\overline{ \frac{ \partial(h_k-h_j)} {\partial{z}_l}}$ of $ \bar {\partial}\left(\overline  {h}_k-\overline  {h}_j\right)$ are conjugate holomorphic on $D(z, \delta/16)$. Using the Cauchy estimate to each of the coefficients,
we deduce that
\begin{align*}
	\left | \overline {\partial}\left(\overline  {h}_k(z) - \overline  {h}_j(z) \right) \right|
	&\lesssim \left( \int_{D(z, \delta/16)} \left | \overline  {h}_k(w) -\overline  {h}_j(w)   \right|^2 dA(w)\right)^{\frac 12}\\
	&\lesssim  G_{\delta/4}(f)(z_k)+   G_{\delta/4}(f)(z_j)\\
 &\lesssim   G_{\delta/2}(f)(z).
\end{align*}
This, together with the identity  $\bar{\partial}\,\overline  {h}_k  = \sum_{j=1}^\infty \psi_j \bar{\partial}(\overline  {h}_k  - \overline  {h}_j ) + H$, yields that for $z\in   D(z_k, 3\delta/16)$,
\begin{align*}
	\left | \bar{\partial}\,\overline  {h}_k(z) \right|^p
	&\lesssim  \left |\sum_{j=1}^\infty \psi_j (z) \bar {\partial}\left(\overline  {h}_k(z) - \overline  {h}_j(z) \right) \right|^p + |H(z)|^p \\
	& \lesssim  \sum_{j,   |z_j-z|<\frac \delta 8} \psi_j (z)\left | \bar {\partial}\left(\overline  {h}_k(z) - \overline  {h}_j(z) \right) \right|^p + |H(z)|^p \\
	& \lesssim  G_{\delta/2}(f)(z)^p + |H(z)|^p.
\end{align*}
Next, we note that for $z\in D(z_k, \delta/8)$, we have $D(z, \delta/16)\subset D(z_k, 3\delta/16)$. Moreover, it follows from the plurisubharmonicity that
\begin{equation}\label{e:PSH}
 |\bar {\partial}\bar  {h}_k (z)|\le  M_{p/2, \delta/16}( \bar{\partial}\bar  {h}_k)(z) \lesssim
     G_{3\delta/4}(f)(z)+ M_{p/2, \delta/16}( H)(z).
\end{equation}
Note that for $z\in \mathbb C$,  there exists some  $w'\in \overline{D(z, \frac{\delta}{16})}$ such that
$$
	M_{2, \delta/16}(H)(z)^p \le \max\left\{ |H(w) |^p: w\in \overline{D(z, \delta/16)}\right\}\\
	=\left|\sum_{k=1}^\infty \psi_k(w')  \bar{\partial} \bar h_k(w')\right|^p.
$$
For  $w'\in \overline{D(z, \frac{\delta}{16})}$, it holds that $G_{3\delta/4}(f)(w') \lesssim G_{\delta}(f)(z)$ and   $M_{p/2, \delta/16}(H)(w')\lesssim M_{p/2, \delta/8}( H)(z)$. Thus, using \eqref{e:PSH}, we obtain
\begin{align}\label{thissss}
	M_{2, \delta/16}(H)(z)^p &\lesssim \sum_{k=1}^\infty \psi_k(w') \left| \bar{\partial} \, \bar h_k(w')\right|^p \nonumber\\
	& \lesssim \sum_{k, \psi_k(w') \neq 0} \psi_k(w')\left(G_{3\delta/4}(f)(w')^p + M_{p/2, \delta/16}( H)(w')^p \right )\nonumber\\
	& \lesssim \sum_{k, \psi_k(w') \neq 0} \psi_k(w')\left(G_{\delta}(f)(z)^p + M_{p/2, \delta/8}( H)(z)^p \right )\nonumber\\
	&= C \left(G_{\delta}(f)(z)^p + M_{p/2, \delta/8}( H)(z)^p \right ).
\end{align}
This shows the required inequality \eqref{claimin}.

By weak-type convolution Young's inequality \cite[Theorem 1.2.13]{MR3243734}, we deduce that
\begin{align}\label{convolution}
\|M_{p/2, \delta/8}( H)\|_{L^{p,\infty}(\mathbb{C})}
=\left\||H|^{\frac{p}{2}} \ast g\right\|_{L^{2,\infty}(\mathbb{C})}^{2/p}
\lesssim \left\||H|^{\frac{p}{2}}\right\|_{L^{2,\infty}(\mathbb{C})}^{2/p}\cdot \left\|g\right\|_{L^1(\mathbb{C})}^{2/p}
\lesssim \|H\|_{L^{p,\infty}(\mathbb{C})},
\end{align}
where $g(z):=\frac{\chi_{D(0,\delta/8)}}{A(D(0,\delta/8))}\in L^1(\mathbb{C})$.
Combining inequalities \eqref{thissss}, \eqref{est-F}, \eqref{f-Sp-norm-b} and \eqref{convolution}, we conclude that
 \begin{align}\label{ppoo}
\|M_{2, \delta/16}(  \bar {\partial}\, \bar{f}_1 )\|_{L^{p,\infty}(\mathbb{C})}
	&\lesssim  \|M_{2, \delta/16} (H)\|_{L^{p,\infty}(\mathbb{C})} +  \|M_{2, \delta/16}( F)\|_{L^{p,\infty}(\mathbb{C})} \nonumber\\
	&\lesssim \|G_{\delta}(f)\|_{L^{p,\infty}(\mathbb{C})} + \|M_{p/2, \delta/8}( H)\|_{L^{p,\infty}(\mathbb{C})}  +  \|M_{2, \delta/16}(G_{\delta/2}(f))\|_{L^{p,\infty}(\mathbb{C})} \nonumber\\
	&\lesssim\|G_{\delta}(f)\|_{L^{p,\infty}(\mathbb{C})}.
 \end{align}
It follows from \cite[Formula (4.7)]{MR4402674} that for $\delta>0$ small enough,
\begin{align}\label{Gdelta}
G_\delta(f)(z)\lesssim \|H_f(k_z)\|_{L^2(\varphi)}.
\end{align}

Similar to the proof of \eqref{cbf1} in the case of $0<p\leq 2$, and to that of \eqref{hff12} in the case of $p>2$, we get that
\begin{align}\label{smito}
\left\|\|H_{\bar{f_1}}(k_z)\|_{L^2(\varphi)}\right\|_{L^{p,\infty}(\mathbb{C})}\lesssim \| M_{2,\delta/16}(\bar{\partial}\bar{f_1})\|_{{L^{p,\infty}(\mathbb{C})}}.
\end{align}
Combining inequalities \eqref{ppoo}, \eqref{Gdelta} and \eqref{smito}, we conclude that
\begin{align*}
\left\|\|H_{\bar{f_1}}(k_z)\|_{L^2(\varphi)}\right\|_{L^{p,\infty}(\mathbb{C})}\lesssim \left\|\|H_{f}(k_z)\|_{L^2(\varphi)}\right\|_{L^{p,\infty}(\mathbb{C})}.
\end{align*}
This, in combination with Theorem \ref{SpandSpinfty} and Proposition \ref{Berezin}, implies that if $H_f\in S^{p,\infty}$, then $H_{\bar{f_1}}\in S^{p,\infty}$.

Next, similar to the proof of \eqref{ssmmt} in the case of $0<p\leq 2$, and to that of \eqref{hff13} in the case of $p>2$, we get that
\begin{align*}
\left\|\|H_{\bar{f_2}}(k_z)\|_{L^2(\varphi)}\right\|_{L^{p,\infty}(\mathbb{C})}\lesssim \| M_{2,\delta/4}(\bar{f_2})\|_{{L^{p,\infty}(\mathbb{C})}}=\| M_{2,\delta/4}(f_2)\|_{{L^{p,\infty}(\mathbb{C})}}.
\end{align*}
By inequality \eqref{decomp-a}, the right-hand side is dominated by $C\|f\|_{IDA^{p,\infty}(\mathbb{C})}$. This, in combination with Theorem \ref{SpandSpinfty}, implies that if $H_f\in S^{p,\infty}$, then $H_{\bar{f_2}}\in S^{p,\infty}$. Since $H_{\bar{f_1}},H_{\bar{f_2}}\in S^{p,\infty}$, we get that $H_{\bar{f}}\in S^{p,\infty}$ and then finish the proof of Theorem \ref{Coburn}.
\end{proof}
	\bigskip

\noindent
 {\bf Acknowledgements:}
Z. Fan is supported by the China Postdoctoral Science Foundation (No. 2023M740799), by the Postdoctoral Fellowship Program of CPSF (No. GZB20230175), and by the Guangdong Basic and Applied Basic Research Foundation (No. 2023A1515110879). X. Wang is supported by  the National Natural Science Foundation of China (Grant No. 12471119 and No. 11971125).

\end{document}